\documentclass[final,leqno]{siamltex}

\usepackage{amsmath}
\usepackage{amssymb}
\usepackage{graphicx}
\usepackage{color}
\usepackage{ifpdf}
\usepackage{url}
\usepackage{hyperref}

\def\norm#1{\|#1\|}

\newcommand{\abs}[1]{\left\vert#1\right\vert}

\newenvironment{algorithm}[1]
{\vskip0.1cm\noindent\textsc{Algorithm } $($#1$)$.}{}
\usecounter{algorithmenumi}

\newcommand{\eo}{\varepsilon_{\mathrm{out}}}
\newcommand{\ei}{\varepsilon_{\mathrm{in}}}

\newcommand{\Lp}{L_{\mathrm{p}}}

\newcommand{\rset}{\mathbb{R}}

\newcommand{\bmtrx}{\left[\begin{array}}
\newcommand{\emtrx}{\end{array}\right]}

\newcommand{\iprod}[1]{\left\langle#1\right\rangle}
\newcommand{\set}[1]{\left\{#1\right\}}

\title{Computational Complexity of Inexact Gradient Augmented Lagrangian Methods:  Application to Constrained MPC \\
{\tiny 10 September 2012}}

\author{Valentin Nedelcu$^{*}$, Ion Necoara\thanks{V. Nedelcu
and I. Necoara are with Automation and Systems Engineering
Department, University Politehnica of Bucharest, 060042 Bucharest,
Romania. {Email:\tt\small\{valentin.nedelcu,
ion.necoara\}@acse.pub.ro}} \and Quoc Tran Dinh\thanks{Q. Tran Dinh
is with Department of Electrical Engineering (ESAT-SCD) and
Optimization in Engineering Center (OPTEC), KU Leuven, Kasteelpark
Arenberg 10, B-3001 Leuven, Belgium. {Email:
\tt\small{quoc.trandinh@esat.kuleuven.be}}}. (Permanent address:
Faculty of Mathematics, Mechanics and Informatics, VNU University of
Science, Hanoi, Vietnam).   }

\begin{document}
\maketitle

\begin{abstract}
We study the  computational complexity certification of inexact
gradient augmented Lagrangian   methods for solving convex
optimization problems  with complicated constraints. We solve the
augmented Lagrangian dual problem that arises from the relaxation of
complicating constraints with  gradient and fast gradient methods
based on inexact first order information. Moreover, since the exact
solution of the augmented Lagrangian primal problem is hard to
compute in practice, we solve this problem up to some given inner
accuracy. We derive relations between the inner and the outer
accuracy of the primal and dual problems and we give a full
convergence rate analysis for both gradient and fast gradient
algorithms.  We  provide  estimates on the  primal and dual
suboptimality and  on  primal feasibility violation of the generated
approximate primal and dual solutions. Our analysis relies on the
Lipschitz property of the dual function and on inexact dual
gradients. We also discuss implementation aspects of the proposed
algorithms on constrained model predictive control problems for
embedded linear systems.
\end{abstract}


\pagestyle{myheadings} \thispagestyle{plain} \markboth{V. Nedelcu,
I. Necoara and Q. Tran Dinh}{Computational complexity of  inexact
gradient augmented Lagrangian methods}

%%%%%%%%%%%%%%%%%%%%%%%%%%%%%%%%%%%%%%%%%%%%%%%%%%%%%%%%%%%%%%%%%%%%%%%%%%%%%%%%%%%%%%%
%%% 1. Introduction
%%%%%%%%%%%%%%%%%%%%%%%%%%%%%%%%%%%%%%%%%%%%%%%%%%%%%%%%%%%%%%%%%%%%%%%%%%%%%%%%%%%%%%%

\section{Introduction} \label{sec_intro}
%Model predictive control (MPC)  is one of  the most successful
%advanced control technology implemented in industry due to its
%ability to handle complex systems with hard input and state
%constraints.  MPC requires the solution of an optimal control
%problem at every sampling instant at which new state information
%becomes available. Recently, embedded MPC has been widely used in
%many applications and its usage in industrial plants has increased
%concurrently. The concept behind embedded MPC is to design a control
%scheme that can be implemented on autonomous electronic hardware,
%e.g a programmable logic controller \cite{ValRos:11}, a
%microcontroller circuit board \cite{Zometa,RicJon:12},
%field-programmable gate arrays  \cite{JerCon:11}, etc. Such devices
%vary widely in both computational power and memory storage
%capabilities as well as cost. As a result, there has been a growing
%focus on developing faster MPC schemes, improving the computational
%efficiency \cite{RaoWri:98} and providing worst case computational
%complexity guarantees for the applied solution methods
%\cite{LanMon:08,RicMor:11,RicJon:12}, making these schemes feasible
%for implementation on cheaper hardware with limited computational
%power.

Embedded control systems has been widely used in many applications
and its usage in industrial plants has increased concurrently. The
concept behind embedded control is to design a control scheme that
can be implemented on autonomous electronic hardware, e.g a
programmable logic controller \cite{ValRos:11}, a microcontroller
circuit board \cite{Zometa,RicJon:12} or field-programmable gate
arrays \cite{JerCon:11}. One of the most successful advanced control
schemes implemented in industry is model predictive control (MPC)
and this is due to its ability to handle complex systems with hard
input and state constraints.  MPC requires the solution of an
optimal control problem at every sampling instant at which new state
information becomes available. In the recent  decades there has been
a growing focus on developing faster MPC schemes, improving the
computational efficiency \cite{RaoWri:98} and providing worst case
computational complexity certificates for the applied solution
methods \cite{KogFin:11,LanMon:08,RicMor:11,RicJon:12}, making these
schemes feasible for implementation on  hardware with limited
computational power.

For fast embedded  systems \cite{HouFer:11,JerCon:11,RicJon:12} the
sampling times are   very short, such that any iterative
optimization algorithm must offer tight bounds on the total number
of iterations which have to be performed in order to provide a
desired optimal controller. Even if second order methods (e.g.
interior point methods) can offer fast rates of convergence in
practice, the worst case complexity bounds are high
\cite{BoyVan:04}. Further,  these methods have complex iterations,
involving inversion of matrices, which are usually difficult to
implement on embedded systems, where the units demand simple
computations. Therefore, first order methods are more suitable in
these situations~\cite{KogFin:11,RicJon:12}.

When the projection on the primal feasible set is hard to compute,
e.g. for constrained MPC problems,  an alternative to primal
gradient methods is to use the Lagrangian relaxation to handle the
complicated constraints and then to apply  dual gradient schemes.
The computational complexity certification of  gradient-based
methods for solving the (augmented) Lagrangian dual of a primal
convex problem  is studied e.g. in
\cite{DevGli:11,KogFin:11,LanMon:08,
NecSuy:08,NedOzd:09,PatBem:12j,RicMor:11,Roc:76}. In
\cite{DevGli:11} the authors present a general framework for
gradient  methods with inexact oracle, i.e. only approximate
information is available for the values of the  function and of its
gradient, and give convergence rate analysis. The authors also apply
their approach to gradient augmented Lagrangian methods and provided
estimates only for dual suboptimality. In \cite{Roc:76} an augmented
Lagrangian algorithm is analyzed  using the theory of monotone
operators. For this algorithm the author proves asymptotic
convergence under general conditions and local linear convergence
under second order optimality conditions. In
\cite{PatBem:12j,PatBem:12} a dual fast gradient method is proposed
for solving quadratic programs with linear inequality constraints
and estimates on primal suboptimality and infeasibility of the
primal solution are provided. In \cite{LanMon:08} the authors
analyze  the iteration complexity of an inexact augmented Lagrangian
method where the approximate solutions of the inner problems are
obtained by using a fast gradient scheme, while the dual variables
are updated by using an inexact dual gradient method. The authors
also provides  upper bounds on the total number of iterations which
have to be performed by the algorithm for obtaining a primal
suboptimal solution. In \cite{NecSuy:08} a dual method based on fast
gradient schemes and smoothing techniques of the ordinary Lagrangian
is presented. Using an averaging scheme the authors are able to
recover a primal suboptimal solution and provide estimates on both
dual and primal suboptimality and also on  primal infeasibility.

%In \cite{RicMor:11}, constrained MPC problems
%with strongly convex quadratic cost and diagonal Hessian are solved
%using Lagrange relaxation and a fast gradient scheme. The authors
%provided estimates for dual suboptimality but there are no results
%regarding the primal suboptimality and feasibility violation, in
%which we are mainly interested in.

Despite widespread use of the dual gradient methods for solving
Lagrangian dual  problems, there are some aspects of these methods
that have not been fully studied. In particular, the previous work
has several limitations. First, the focus is mainly on the
convergence analysis of the dual variables. Second, only the dual
gradient method is usually analyzed and using exact information.
Third, there is no full convergence rate analysis (i.e. no estimates
in terms of dual and primal suboptimality and primal feasibility
violation) for both dual gradient and fast gradient schemes, while
using inexact dual information. Therefore, in this paper we focus on
solving convex optimization problems (possibly nonsmooth)
approximately by using an augmented Lagrangian approach and inexact
dual gradient and fast gradient methods. We show how approximate
primal solutions can be  generated based on averaging for general
convex  problems and we give a full convergence rate analysis for
both algorithms that leads to error estimates on the amount of
constraint violation and the cost of primal and dual solutions.
Since we allow one to solve the inner problems approximately, our
dual gradient schemes have to use inexact information.

%%% a. Contribution.
\paragraph{Contribution} The contributions of this paper include the following:
\begin{enumerate}
\item We propose and analyze  dual  gradient  algorithms producing approximate primal
 feasible and optimal solutions.  Our analysis is based on the
augmented Lagrangian framework which leads to the dual function
having Lipschitz continuous gradient, even if the primal objective
function is not strongly convex.

\item Since  exact  solutions of the inner problems are  usually hard to compute,
 we solve these problems only up to a certain inner accuracy $\ei$.
We analyze several stopping criteria which can be used in order to find such a solution
and point out their advantages.

\item For solving  outer problem we propose two inexact dual gradient
algorithms:\\
- an inexact dual gradient algorithm, with  complexity
$\mathcal{O}(1/\eo)$ iterations, which allows us to  find an
$\eo$-optimal solution of the original problem by solving the inner
problems with
an accuracy $\ei$ of order $\mathcal{O}(\eo)$.\\
- an inexact dual fast gradient algorithm, with complexity
$\mathcal{O}(\sqrt{1/\eo})$ iterations, provided that the  inner
problems are solved with  accuracy $\ei$ of order $\mathcal{O}(\eo
\sqrt{\eo})$.

\item For both methods we show how to generate approximate primal solutions and
 provide estimates on  dual and primal suboptimality and  primal
infeasibility.

\item To certify  the complexity of the proposed methods, we apply the algorithms
on linear embedded MPC problems with state and input constraints.
\end{enumerate}

%% b. Paper outline.
\paragraph{Paper outline}
The paper is organized as follows. In Section \ref{sec_intro},
motivated by embedded  MPC, we introduce   the augmented Lagrangian
framework for solving constrained convex  problems. In Section
\ref{sec_inner} we discuss different stopping criteria for finding a
suboptimal solution of the inner problems and provide estimates on
the complexity of finding such a solution. In Section
\ref{sec_outer} we propose an inexact dual gradient and fast
gradient algorithm for solving the outer problem. For both
algorithms we provide   bounds on the dual and primal suboptimality
and also on the primal infeasibility. In Section \ref{sec_MPC} we
specialize our general results  to  constrained linear MPC problems
and we obtain tight bounds on the number of inner and outer
iterations. We also provide extensive numerical tests to prove the
efficiency of the proposed algorithms.

%% b. Notation
\paragraph{Notation and terminology}
We work in the space $\rset^n$ composed by column vectors. For $x, y
\in \rset^n$,  $\iprod{x, y} := x^Ty = \sum_{i=1}^n x_i y_i$ and
$\norm{x} :=  (\sum_{i=1}^n x_i^2)^{1/2}$  denote the standard
Euclidean inner product and  norm, respectively. We use the same
notation $\iprod{\cdot, \cdot}$ and $\norm{\cdot}$ for spaces of
different dimension.  We denote by $\mathrm{cone}\{a_i, ~ i \in I
\}$ the cone generated from vectors $\{a_i, ~ i\in I\}$. We also
denote by $R_{\text{p}} := \max_{z,y \in Z} \norm{z-y}$ the
diameter, $\mathrm{int}(Z)$ the interior and $\mathrm{bd}(Z)$ the
boundary of a convex, compact set $Z$. By $\mathrm{dist}(y, Z)$ we
denote the Euclidean distance from a point $y$ to the set $Z$ and by
$h_Z(y) := \sup_{z \in Z} y^T z$ the support function of the set
$Z$. For any point $\tilde{z} \in Z$ we denote by
$\mathcal{N}_Z(\tilde{z}) :=\set{s ~|~ \iprod{s, z-\tilde{z}} \leq 0
~\forall z \in Z}$ the normal cone of $Z$ at $\tilde{z}$. For a real
number $x$, $\lfloor{x}\rfloor$ denotes the largest integer number
which is less than or equal to $x$, while ``$:=$'' means
``equal by definition''. %We say that a quantity $q(\cdot)$ is of
%complexity order $\mathcal{O}(g(\epsilon))$, with $g(\epsilon)>0$,
%if for some positive scalars $c$ we have $q(\cdot)\leq c
%g(\epsilon)$.

%%%%%%%%%%%%%%%%%%%%%%%%%%%%%%%%%%%%%%%%%%%%%%%%%%%%%%%%%%%%%%%%%%%%%%%%%%%%%%%
%% 1.1. Motivating example.

\subsection{A motivating example: Linear MPC problems with state-input constraints}
\label{subsec_motivation} We consider a discrete time linear system
given by the dynamics:
\begin{equation*}
x_{k+1}= A_x x_k + B_u u_k,
\end{equation*}
where $x_k \in \rset^{n_x}$ represents the state and $u_k \in
\rset^{n_u}$ represents  the input of the system.  We also assume
hard state and input constraints:
\begin{equation*}
x_k \in X \subseteq \rset^{n_x}, \quad u_k \in U \subseteq
\rset^{n_u} \;\; \forall k \geq 0.
\end{equation*}
Now, we can define the linear MPC problem over the prediction
horizon of length $N$,  for a given initial state $x$, as follows
\cite{ScoMay:99}:
\begin{equation}\label{eq:LMPC}
f^{*}(x) := \left\{ \begin{array}{cl}
\displaystyle\min_{x_i, u_i}&\sum_{i=0}^{N-1}\ell(x_i,u_i) + \ell_{\mathrm{f}}(x_N) \\
\textrm{s.t.} &x_{i+1} = A_x x_i + B_u u_i, ~ x_0 = x,\\
&x_i \in X, ~u_i \in U ~\forall i, ~ x_N \in X_{\mathrm{f}},
\end{array}\right.
\end{equation}
where both the stage cost $\ell$ and the terminal cost
$\ell_{\mathrm{f}}$ are convex  functions (possibly nonsmooth). Note
that in our formulation we do not require strongly convex  costs.
Further, the terminal set $X_{\mathrm{f}}$ is chosen so that
stability of the closed-loop system is guaranteed.   We  assume the
sets $X, U$ and $X_{\mathrm{f}}$ to be compact, convex and simple
(by simple we understand that the projection on these sets can be
done easily, e.g. boxes).

Furthermore, we introduce the notation $z :=  \left[x_1^T \cdots
x_N^T \; u_0^T \cdots u_{N-1}^T \right]^T$, $Z := \prod_{i=1}^{N-1}
X \times X_{\mathrm{f}}\times \prod_{i=1}^N U$ and $f(z)  :=
\sum_{i=0}^{N-1} \ell(x_i,u_i) + \ell_{\mathrm{f}}(x_N)$. We can
also write compactly the linear dynamics $ x_{i+1} = A_x x_i + B_u
u_i$ for all $i = 0, \cdots, N-1$ and $x_0 = x$ as $Az = b(x)$ (see
\cite{ScoMay:99,WanBoy:10} for details). Note that $b(x)\in \rset^{N
n_x}$ depends linearly on $x$, i.e. $b(x) := \left[(A_x x)^T \ 0^T
\cdots 0^T \right]^T$. In these settings, for linear MPC we need to
solve, for a given initial state $x$, the primal convex optimization
problem:
\begin{equation}\label{eq:Px}
\min_z\set{f(z)~|~ Az= b(x), ~z \in Z}, \tag{$\textbf{P}(x)$}
\end{equation}
where $f$ is a convex function (possibly nonsmooth) and $A$ is a
matrix of appropriate dimension. Moreover, the set $Z$ is simple as
long as $X$, $X_{\mathrm{f}}$ and $U$ are simple sets. In the
following sections we  discuss how we can efficiently solve
optimization problem \eqref{eq:Px} approximately with dual gradient
methods based on inexact first order information and we  provide
tight estimates for the total number of iterations which has to be
performed in order to obtain a suboptimal solution in terms of
primal suboptimality and infeasibility.

%%%%%%%%%%%%%%%%%%%%%%%%%%%%%%%%%%%%%%%%%%%%%%%%%%%%%%%%%%%%%%%%%%%%%%%%%%%%%%%%%%%%%%%%%%%%%%
%% 1.2. Augmented Lagrangian framework.

\subsection{Augmented Lagrangian framework}\label{subsec_Aug}
Motivated by MPC problems, we are interested in solving convex
optimization problems of the form:
\begin{align*}\label{eq:primal_prob}
f^{*} := \left\{\begin{array}{cl}
\displaystyle\min_{z\in \rset^n} &f(z)\\
\textrm{s.t.} &Az = b, ~z \in Z,
\end{array}\right.\tag{$\textbf{P}$}
\end{align*}
where $f$ is convex function (possibly nonsmooth), $A \in \rset^{m
\times n}$ is  a full row-rank matrix and $Z$ is  a simple set (i.e.
the projection on this set is computationally cheap), compact and
convex. We will denote problem \eqref{eq:primal_prob}  as the primal
problem and $f$ as the primal objective function.

A common approach for solving problem \eqref{eq:primal_prob}
consists of applying  interior point methods, which usually perform
much lower number of iterations in practice than those predicted by
the theoretical worst case complexity analysis \cite{BoyVan:04}.  On
the other hand, for first order methods the number of iterations
predicted by the worst case complexity analysis is close to the
actual number of iterations performed by the method \cite{Nes:04}.
This is crucial in the context of fast embedded systems. First order
methods applied directly to problem \eqref{eq:primal_prob} imply
projection on the feasible set $\set{z ~|~ z \in Z, ~Az = b}$. Note
that even if $Z$ is a simple set, the projection on the feasible set
is hard due to the complicating constraints $Az = b$. An efficient
alternative is to move the complicating constraints into the cost
via Lagrange multipliers and solve the dual problem approximately by
using first order methods and then recover a primal suboptimal
solution for \eqref{eq:primal_prob}. This is the approach that we
follow in this paper: we derive inexact dual gradient methods that
allow us to generate approximate primal solutions for which we
provide  estimates for the violation of the constraints and upper
and lower bounds on the corresponding primal objective function
value of \eqref{eq:primal_prob}.

First let us define the dual function:
\begin{equation}\label{eq:dual_fun}
d(\lambda) := \min_{z \in Z} \mathcal{L}(z,\lambda),
\end{equation}
where $\mathcal{L}(z,\lambda) := f(z) +\iprod{\lambda, Az - b}$
represents the partial  Lagrangian  with respect to the constraints
$Az = b$ and $\lambda$ the associated Lagrange multipliers. Now, we
can write the corresponding dual problem as follows:
\begin{equation}\label{eq:dual_prob}
\max_{\lambda\in\rset^m} d(\lambda). \tag{$\textbf{D}$}
\end{equation}
We assume that Slater's constraint qualification holds, so that
problems \eqref{eq:primal_prob}  and \eqref{eq:dual_prob} have the
same optimal value. We also denote by $z^{*}$ an optimal solution of
\eqref{eq:primal_prob} and by $\lambda^{*}$ the corresponding
multiplier (i.e. an optimal solution of \eqref{eq:dual_prob}).

In general, the dual function $d$ is not differentiable
\cite{Ber:99} and therefore any subgradient  method for solving
\eqref{eq:dual_prob} has a slow convergence rate \cite{NedOzd:09}.
We will see in the sequel how we can avoid this drawback by means of
augmented Lagrangian framework.  We define the augmented Lagrangian
function~\cite{Hes:69}:
\begin{equation}\label{eq:aug_Lag_fun}
\mathcal{L}_{\rho}(z,\lambda) := f(z) + \iprod{\lambda, Az - b} +
\frac{\rho}{2}\norm{Az - b}^2,
\end{equation}
where $\rho > 0$ represents a penalty parameter. The augmented
dual problem, called also  the \emph{outer} problem, is defined
as:
\begin{equation}\label{eq:aug_dual_prob}
\max_{\lambda \in \rset^m} d_{\rho}(\lambda),
\tag{$\textbf{D}_{\rho}$}
\end{equation}
where $d_{\rho}(\lambda) := \displaystyle\min_{z \in Z}
\mathcal{L}_{\rho}(z,\lambda)$ and  we denote by $z^{*}(\lambda)$
an optimal solution of the \emph{inner} problem $\min_{z \in Z}
\mathcal{L}_{\rho}(z, \lambda)$ for a given $\lambda$. It is
well-known \cite{Ber:99,LanMon:08} that  the optimal value and the
set of optimal solutions of the dual problems \eqref{eq:dual_prob}
and \eqref{eq:aug_dual_prob} coincide.  Furthermore, the function
$d_{\rho}$ is concave and differentiable and its gradient is
\cite{Nes:04a}:
\begin{equation*}
\nabla{d}_{\rho}(\lambda) := Az^{*}(\lambda) - b.
\end{equation*}
Moreover, the gradient mapping $\nabla{d}_{\rho}(\cdot)$ is
Lipschitz continuous with  a Lipschitz constant \cite{Ber:99}
given by:
\begin{equation*}
L_{\mathrm{d}} := \rho^{-1}.
\end{equation*}
In conclusion, we want to solve within an accuracy $\eo$ the
equivalent  smooth outer problem \eqref{eq:aug_dual_prob} by using
first order methods with inexact gradients (e.g. dual gradient or
fast gradient algorithms) and then recover an approximate primal
solution. In other words, the goal of this paper is to generate a
primal-dual pair $(\hat{z}, \hat{\lambda})$, with $\hat{z} \in Z$,
for which we can ensure bounds on  dual suboptimality,  primal
infeasibility and primal suboptimality of order $\eo$, i.e.:
\begin{equation}
\label{eq:out1} f^{*} - d_{\rho}(\hat{\lambda}) \leq
\mathcal{O}\left(\eo\right), ~\norm{A \hat{z} - b} \leq
\mathcal{O}\left(\eo\right) ~ \textrm{and}~\abs{f(\hat{z}) -
f^{*}} \leq \mathcal{O}\left(\eo\right).
\end{equation}

%%%%%%%%%%%%%%%%%%%%%%%%%%%%%%%%%%%%%%%%%%%%%%%%%%%%%%%%%%%%%%%%%%%%%%%%%%%%%%%%
%% II. Complexity estimates of solving the inner problem.
%%%%%%%%%%%%%%%%%%%%%%%%%%%%%%%%%%%%%%%%%%%%%%%%%%%%%%%%%%%%%%%%%%%%%%%%%%%%%%%%

\section{Complexity estimates of solving the inner problems}\label{sec_inner}
As we have seen in the previous section, in order to compute the
gradient $\nabla{d}_{\rho}$  we have to find, for a given $\lambda$,
an optimal solution of the \textit{inner} convex problem:
\begin{equation}\label{eq:inner_problem}
z^*(\lambda) \in \arg\min_{z \in Z} \mathcal{L}_{\rho} (z,\lambda).
%\set{ f(z) + \iprod{\lambda, Az - b} +  \frac{\rho}{2}\norm{Az - b}^2}.
\end{equation}
From the optimality conditions \cite{Roc:98}, we know that a point
$z^{*}(\lambda)$ is an optimal solution of
\eqref{eq:inner_problem} if and only if:
\begin{equation}
\label{eq:optimality_conditions}
\iprod{\nabla{\mathcal{L}}_{\rho}\left(z^{*}(\lambda),\lambda\right),
z-z^{*}(\lambda)} \geq 0 ~~~  \forall z \in Z.
\end{equation}
An equivalent way to characterize an optimal solution
$z^{*}(\lambda)$ of  \eqref{eq:inner_problem} can be given in terms
of the following inclusion:
\begin{equation}\label{eq:optimality_cone}
0 \in
\nabla{\mathcal{L}}_{\rho}\left(z^{*}(\lambda),\lambda)+\mathcal{N}_Z(z^{*}(\lambda)\right).
\end{equation}
Since an exact minimizer of the inner problem
\eqref{eq:inner_problem}  is usually hard to compute, we are
interested in finding an approximate solution of this problem
instead of its optimal one. Therefore, we have to consider an
inner accuracy $\varepsilon_{\mathrm{in}}$ which measures the
suboptimality of such an approximate solution for
\eqref{eq:inner_problem}:
\begin{equation*}
\bar{z}(\lambda) \approx \arg\min_{z \in Z}\set{ f(z) +
\iprod{\lambda, Az - b}  + \frac{\rho}{2}\norm{Az - b}^2}.
\end{equation*}
Since there exist several ways  to characterize an $\ei$-optimal
solution \cite{DevGli:11,LanMon:08,Roc:76}, we will further discuss
 different stopping criteria which can be used in order to
find such a solution. A well-known stopping criterion, which
measures the distance to optimal value of
\eqref{eq:inner_problem}, is given by:
\begin{equation}
\label{eq:criterion_a} \bar{z}(\lambda) \in Z, \;
~\mathcal{L}_{\rho}\left(\bar{z}(\lambda),  \lambda\right) -
\mathcal{L}_{\rho}\left(z^{*}(\lambda), \lambda\right) \leq
 \ei^2.
\end{equation}
The main advantage of using \eqref{eq:criterion_a} as a stopping
criterion for finding  $\bar{z}(\lambda)$ consists of the fact that
in the literature \cite{Nes:04} there exist explicit bounds on the
number of iterations which has to be performed by some well-known
first or second order methods to ensure the $\ei$-optimality.

Another stopping criterion, which measures the distance of
$\bar{z}(\lambda)$  to the set of optimal solution $Z^*(\lambda)$
of \eqref{eq:inner_problem} is given by:
\begin{equation}\label{eq:criterion_b}
\bar{z}(\lambda) \in Z, ~\mathrm{dist} \left(\bar{z}(\lambda),
Z^{*}(\lambda)\right) \leq \mathcal{O} \left( \ei \right).
\end{equation}
It is known  that this distance can be bounded by an easily
computable quantity when the objective function satisfies the
so-called gradient error bound \cite{Ber:99}. Thus, we can use this
bound  to define stopping rules in iterative algorithms for solving
the optimization problem. Note that gradient error bound assumption
is a generalization  of the more restrictive  notion of strong
convexity.

As a direct consequence of the optimality condition
\eqref{eq:optimality_conditions},  one can use the following
stopping criterion:
\begin{equation}
\label{eq:criterion_c} \bar{z}(\lambda) \in Z,
~~\iprod{\nabla{\mathcal{L}}_{\rho} (\bar{z}(\lambda),\lambda), z -
\bar{z}(\lambda)} \geq -\mathcal{O}(\ei) ~~~\forall z \in Z.
\end{equation}
Note that \eqref{eq:criterion_c} can be formulated using the support
function as:
\begin{equation*}
h_Z\left(-\nabla{\mathcal{L}}_{\rho}\left(\bar{z}(\lambda),\lambda\right)\right)
+
\iprod{\nabla{\mathcal{L}}_{\rho}(\bar{z}\left(\lambda),\lambda\right),\bar{z}(\lambda)}\leq
\mathcal{O}\left(\ei\right).
\end{equation*}
When the set $Z$ has a specific structure (e.g. a ball defined by
some norm), tight  upper bounds on the support function can be
computed explicitly and thus the stopping criterion can be
efficiently verified.

Based on optimality conditions \eqref{eq:optimality_cone}, the
following stopping criterion can also be used in order to
characterize an $\ei$-optimal solution $\bar{z}(\lambda)$ of the
inner problem \eqref{eq:inner_problem}:
\begin{equation}\label{eq:inner_criterion}
 \bar{z}(\lambda) \in Z, ~\mathrm{dist} \big(0,\nabla \mathcal{L}_{\rho}(\bar{z}(\lambda),\lambda)
 + \mathcal{N}_Z(\bar{z}(\lambda))\big) \leq \mathcal{O}(\ei).
\end{equation}
The main advantage of using this criterion is given by the fact that
the distance in \eqref{eq:inner_criterion} can be computed
efficiently for  sets $Z$ having a certain structure. Note that
\eqref{eq:inner_criterion} can be verified by solving the following
projection problem:
\begin{equation}\label{eq_distance}
s^{*} \in \arg\!\!\!\!\!\!\!\!\min_{s \in \mathcal{N}_Z\left(\bar{z}(\lambda)\right)}\norm{\nabla{\mathcal{L}}_{\rho}\left(\bar{z}(\lambda),\lambda\right) +
s}^2.
\end{equation}
To see how \eqref{eq_distance} can be solved efficiently we are first interested in finding an explicit characterization of the normal cone
$\mathcal{N}_Z(\bar{z}(\lambda))$, when the set $Z$ has a certain structure.

%% Lemma 2.2.
\begin{lemma}\label{le:special_case_inner_prob}{~}
Assume that the set $Z$ is a general polyhedral set, i.e. $Z :=
\set{z \in \rset^n~|~ C z \leq c}$, with $C \in \rset^{p \times
n}$ and $c \in \rset^p$. Then, problem \eqref{eq_distance} can be
recast as the following quadratic optimization problem:
\begin{equation}\label{eq_distance2}
\min_{\mu \geq 0}\norm{\nabla{\mathcal{L}}_{\rho}\left(\bar{z}(\lambda),\lambda\right) + \tilde{C}^T\mu}^2,
\end{equation}
where matrix $\tilde{C}$ contains the rows of $C$  corresponding to
the active constraints in $C \bar{z}(\lambda) \leq c$. In
particular, if $Z$ is a box in $\mathbb{R}^n$, then  problem
\eqref{eq_distance2} becomes separable and it can be solved
explicitly in $\mathcal{O}(\tilde{p})$ operations, where $\tilde{p}$
represents the number of active constraints in $C \bar{z}(\lambda)
\leq c$.
\end{lemma}

%%% Proof of Lemma 2.2.
\begin{proof}
Let us recall that if $\bar{z}(\lambda) \in \mathrm{int} (Z)$, then
we have  $\mathcal{N}_Z\left(\bar{z}(\lambda)\right) = \set{0}$ and
therefore the distance $\mathrm{dist}\left(0,\nabla
\mathcal{L}_{\rho}\left(\bar{z}(\lambda),\lambda\right) +
\mathcal{N}_Z\left(\bar{z}(\lambda)\right)\right)$ will be equal to
$\norm{\nabla{\mathcal{L}}_{\rho}\left(\bar{z}(\lambda),\lambda\right)}$.
In the case of  $\bar{z}(\lambda) \in \mathrm{bd}(Z)$ there exists
an index set $\mathrm{I}\left(\bar{z}(\lambda)\right) \subseteq
\set{1,\cdots, p}$ such that $C_i  \bar{z}(\lambda) = c_i$ for all
$i \in \mathrm{I}(\bar{z}(\lambda))$, where $C_i$ and $c_i$
represent the $i$-th row and $i$-th element of $C$ and $c$,
respectively. Using now Theorem 6.46 in \cite{Roc:98} we have
$\mathcal{N}_Z\left(\bar{z}(\lambda)\right) =
\mathrm{cone}\set{C_i^T ,i \in
\mathrm{I}\left(\bar{z}(\lambda)\right)}$. Introducing the notation
$\tilde{C}$ for the matrix whose rows are $C_i$ for all $i \in
\mathrm{I}\left(\bar{z}(\lambda)\right)$,  we can write
\eqref{eq_distance} as \eqref{eq_distance2}. Note  that, in problem
\eqref{eq_distance2}, the dimension of the variable $\mu$ is
$\tilde{p} = \abs{\mathrm{I}\left(\bar{z}(\lambda)\right)}$ (i.e.
the number of active constraints) which usually is smaller in
comparison with the dimension $n$ of  problem \eqref{eq_distance}.

Now, if we assume that $Z$ is a box in $\rset^n$, then problem
\eqref{eq_distance2} can be written in the following equivalent
form:
\begin{equation*}
\min_{\mu \geq 0} \frac{1}{2}\mu^T \tilde{C} \tilde{C}^T \mu +
\nabla \mathcal{L}_{\rho}^T\left(\bar{z}(\lambda),\lambda\right)\tilde{C}^T \mu.
\end{equation*}
Since for box constraints we have $\tilde{C} \tilde{C}^T =
\mathrm{I}_{\tilde{p}}$, i.e. the identity matrix, the previous
optimization problem can be decomposed into $\tilde{p}$ scalar
projection problems onto the nonnegative orthant and thus in order
to compute the optimal solution of \eqref{eq_distance2} we only have
to perform $\mathcal{O}(\tilde{p})$ comparisons.
\end{proof}
%% End of the proof.

The next lemma establishes some relations between  stopping criteria
\eqref{eq:criterion_a}-\eqref{eq:inner_criterion}.
%% Lemma 2.3.
\begin{lemma}\label{le:inexact_criterions}
The conditions \eqref{eq:criterion_a}, \eqref{eq:criterion_b},
\eqref{eq:criterion_c} and \eqref{eq:inner_criterion} satisfy the
following:
\begin{itemize}
\item[\emph{(i)}] Let $\nabla{\mathcal{L}}_\rho$ be Lipschitz
continuous with a Lipschitz constant $L_{\mathrm{p}}$. Then
\begin{equation*}
\eqref{eq:criterion_a} ~\Rightarrow ~\eqref{eq:criterion_c},
~\eqref{eq:criterion_b} ~\Rightarrow ~\eqref{eq:criterion_c},
~\eqref{eq:inner_criterion} ~\Rightarrow ~\eqref{eq:criterion_c}.
\end{equation*}
\item[\emph{(ii)}] If, in addition, $\mathcal{L}_{\rho}$ is strongly convex with a convexity parameter $\sigma_{\text{p}} >0 $, then
\begin{equation*}
\eqref{eq:inner_criterion} ~\Rightarrow
~\eqref{eq:criterion_a}~\Rightarrow~\eqref{eq:criterion_b}.
\end{equation*}
\end{itemize}
\end{lemma}

%% Proof of Lemma 2.3.
\begin{proof}
\emph{(i)}
\eqref{eq:criterion_a}~$\Rightarrow$~\eqref{eq:criterion_c}: In
Section 3 of \cite{DevGli:11} the authors show that if
\eqref{eq:criterion_a} holds, then \eqref{eq:criterion_c} also
holds with
$\mathcal{O}(\ei)=\ei^2+\sqrt{2L_{\mathrm{p}}}R_\mathrm{p}\ei$.

\eqref{eq:criterion_b}~$\Rightarrow$~\eqref{eq:criterion_c}: We
can write:
\begin{align*}
&\iprod{\nabla{\mathcal{L}}_{\rho}\right(\bar{z}(\lambda),\lambda\left),
z-\bar{z}(\lambda)} \\
&=\! \iprod{\nabla{\mathcal{L}}_{\rho} \!
\right(\bar{z}(\lambda),\lambda\left)
 - \nabla{\mathcal{L}}_{\rho} \! \left(z^{*}(\lambda),\lambda\right), z
\!-\! \bar{z}(\lambda)} \!+\! \iprod{\nabla{\mathcal{L}}_{\rho} \!
\left(z^{*}(\lambda),\lambda\right), z \!-\! z^{*}
(\lambda) \!+\! z^{*}(\lambda) - \bar{z}(\lambda)} \\
&\geq -\left( L_{\text{p}}R_{\text{p}} + \norm{\nabla{\mathcal{L}}_{\rho}\left(z^{*}(\lambda),\lambda\right)} \right) \ei.
\end{align*}
Since $Z$ is compact and $\nabla{\mathcal{L}}_{\rho}(\cdot,\lambda)$
is continuous, then $\nabla{\mathcal{L}}_{\rho}(\cdot,\lambda)$ is
bounded. Hence,  our statement follows from the last inequality.
\newline
\eqref{eq:inner_criterion}~$\Rightarrow$~\eqref{eq:criterion_c}: The
condition \eqref{eq:inner_criterion} can be written as
$\iprod{\nabla{\mathcal{L}}_{\rho}\right(\bar{z}(\lambda),\lambda\left),
z -\bar{z}(\lambda)} \geq - \ei \iprod{e, z -\bar{z}(\lambda)}$ for
all $e$ such that $\norm{e}\leq 1$ and $z\in Z$. If $Z$ is bounded,
then the last inequality implies
$\iprod{\nabla{\mathcal{L}}_{\rho}\right(\bar{z}(\lambda),\lambda\left),
z -\bar{z}(\lambda)} \geq - R_{\text{p}} \ei$.
\newline
\emph{(ii)}
\eqref{eq:inner_criterion}~$\Rightarrow$~\eqref{eq:criterion_a}:
Since $\mathcal{L}_{\rho}(z,\lambda)$ is strongly convex and
$z^{*}(\lambda)$ is its minimizer over $Z$,   we have:
\begin{equation*}
0 \geq \mathcal{L}_{\rho}\left(\bar{z}(\lambda),\lambda\right)-
\mathcal{L}_{\rho}\left(z^{*}(\lambda),\lambda\right) \geq
\frac{\sigma_{\mathrm{p}}}{2}\|\bar{z}(\lambda)-z^{*}(\lambda)\|^2.
\end{equation*}
From the convexity of $\mathcal{L}_{\rho}(z,\lambda)$ we can write:
\begin{align*}
& \mathcal{L}_{\rho}\!\left(\bar{z}(\lambda),\lambda\right) -
\mathcal{L}_{\rho}\left(z^{*}(\lambda),\lambda\right) \leq \iprod{
\nabla{\mathcal{L}}_{\rho} \left(\bar{z}(\lambda), \lambda\right),
\bar{z}(\lambda)  - z^{*}(\lambda)}\\
&\leq \iprod{\nabla{\mathcal{L}}_{\rho}\left(\bar{z}(\lambda),
\lambda\right)+s^{*}, \bar{z}(\lambda)-z^{*}(\lambda)} \leq
\norm{\nabla \mathcal{L}_{\rho}\left(\bar{z}(\lambda),
\lambda\right)+s^{*}}\norm{\bar{z}(\lambda)-z^{*}(\lambda)}\\
&\leq \ei \left[\frac{2}{\sigma_{\mathrm{p}}}\left(
\mathcal{L}_{\rho}\left(\bar{z}(\lambda),\lambda\right)-
\!\mathcal{L}_{\rho}\left(z^{*}(\lambda),\lambda\right)\right)\right]^{1/2},
\end{align*}
which implies \eqref{eq:criterion_a}.
\newline
\eqref{eq:criterion_a}~$\Rightarrow$~\eqref{eq:criterion_b}:  Taking
into account that $\mathcal{L}_{\rho}(\cdot, \lambda)$  is strongly
convex we have
$\frac{\sigma_{\mathrm{p}}}{2}\norm{z^{*}(\lambda)-\bar{z}(\lambda)}
\leq \mathcal{L}_{\rho}\left(\bar{z}(\lambda),\lambda\right) -
\mathcal{L}_{\rho}\left(z^{*}(\lambda),\lambda\right) \leq \ei^2$.
This leads to $\norm{\bar{z}(\lambda) - z^{*}(\lambda)}\leq
(2/\sigma_{\text{p}})^{1/2}\ei$. The lemma is proved.
\end{proof}
%% End of the proof.

The next theorem provides estimates on the number of iterations that
are required by  fast gradient schemes to obtain an $\ei$
approximate solution for  inner problem \eqref{eq:inner_problem}.

%%% Theorem 2.3.
\begin{theorem}\cite{Nes:04}\label{th:inner_complexity}
Assume that function $\mathcal{L}_{\rho}(\cdot,\lambda)$ has
Lipschitz continuous gradient  w.r.t. variable $z$,  with a
Lipschitz constant $L_{\text{p}}$ and a fast gradient scheme
\cite{Nes:04} is applied for finding an $\ei$ approximate solution
$\bar{z}(\lambda)$ of \eqref{eq:inner_problem} such that stopping
criterion \eqref{eq:criterion_a} holds, i.e. $\mathcal{L}_{\rho}
\left( \bar{z}(\lambda),\lambda \right)-\mathcal{L}_{\rho}
\left(z^*(\lambda),\lambda \right) \leq \ei^2$. Then, the complexity
of finding $\bar{z}(\lambda)$ is  $\mathcal{O} \left(
\sqrt{\frac{L_p}{\ei^2}} \right)$ iterations. If, in addition
$\mathcal{L}_{\rho}(\cdot,\cdot)$ is strongly convex with a
convexity parameter $\sigma_{\text{p}} > 0$, then $\bar{z}(\lambda)$
can be computed in at most
$\mathcal{O}\left(\sqrt{\frac{L_\text{p}}{\sigma_{\text{p}}}} \ln
\left(\frac{\sigma_{\text{p}}}{\ei^2} \right) \right)$ iterations by
using a fast gradient scheme.
\end{theorem}

Note that if the function $f$ is nonsmooth, we have a complexity
$\mathcal{O} \left( \frac{1}{\left(\ei^2\right)^2} \right)$
iterations with a projected subgradient method or an improved
$\mathcal{O} \left(\frac{1}{\ei^2} \right)$ by using  smoothing
techniques \cite{Nes:04a}, provided that $f$ has a certain
structure.

%%%%%%%%%%%%%%%%%%%%%%%%%%%%%%%%%%%%%%%%%%%%%%%%%%%%%%%%%%%%%%%%%%%%%%%%%%%%%%%%%%%%%%%%%%%
%% III. Complexity estimates of solving the outer problem using approximate dual gradients
%%%%%%%%%%%%%%%%%%%%%%%%%%%%%%%%%%%%%%%%%%%%%%%%%%%%%%%%%%%%%%%%%%%%%%%%%%%%%%%%%%%%%%%%%%%

\section{Complexity estimates of solving the outer problem using approximate dual gradients}
\label{sec_outer}
In this section we solve the augmented Lagrangian dual problem
\eqref{eq:aug_dual_prob} approximately by using dual gradient and
fast gradient methods with inexact information and derive
computational complexity certificates for these methods. Since we
solve the inner problem inexactly, we have to use inexact gradients
and approximate values of the augmented dual function $d_{\rho}$
defined in terms of $\bar{z}(\lambda)$, i.e. we introduce the
following pair:
\begin{equation*}
\bar{d}_{\rho}(\lambda):= \mathcal{L}_{\rho}\left(\bar{z}(\lambda),\lambda\right) ~~\textrm{and}~~\nabla{\bar{d}}_{\rho}(\lambda) := A\bar{z}(\lambda) - b.
\end{equation*}
The next  theorem, which is similar to the results in
\cite{DevGli:11}, provides  bounds on the dual function when the
inner problem \eqref{eq:inner_problem} is solved approximately. For
completeness we  give the proof.

%% Theorem 3.1.
\begin{theorem}\label{th:cond_echiv}
If $\bar{z}(\lambda)$ is computed such that the stopping criterion
\eqref{eq:criterion_c} is satisfied, i.e. $\bar{z}(\lambda) \in Z$
 and $\min_{z \in Z}
\iprod{\nabla{\mathcal{L}}_{\rho}(\bar{z}(\lambda), \lambda), z -
\bar{z}(\lambda)} \geq
-\left(1+\sqrt{2L_{\mathrm{p}}}R_{\mathrm{p}}\right) \ei$, then
the following inequalities hold:
\begin{eqnarray}\label{ineq_approx}
&&\bar{d}_{\rho}(\lambda) +  \iprod{\nabla {\bar{d}_{\rho}}(\lambda),  \mu  -\lambda}  - \frac{L_{\mathrm{d}}}{2}\norm{\mu - \lambda}^2 - \left(
1 + \sqrt{2L_{\mathrm{p}}}R_{\text{p}}\right)\ei \leq d_{\rho}(\mu) \nonumber\\
[-1.5ex]\\[-1.5ex]
&& \leq \bar{d}_{\rho}(\lambda) +
\iprod{\nabla{\bar{d}_{\rho}}(\lambda), \mu  - \lambda} ~~~ \forall
\lambda, \mu \in \rset^m.\nonumber
\end{eqnarray}
\end{theorem}

%% Proof of Theorem 3.1.
\begin{proof}
For simplicity we introduce the  notation
$C_Z:=1+\sqrt{2\Lp}R_{\text{p}}$. The right-hand side inequality
follows directly from the definitions of $d_{\rho}$ and
$\bar{d}_\rho$. We only have to prove the left-hand side
inequality. Following the derivations from Section 3.3 in
\cite{DevGli:11} we have:
\begin{align*}%\label{ineq_interm}
d_{\rho}(\mu) & \geq \min_{z \in Z}\Big\{ f(\bar{z}(\lambda))+
\iprod{\nabla{f}(\bar{z}(\lambda)), z\!-\!\bar{z}(\lambda)} +
\iprod{\mu, A z - b} +
\frac{\rho}{2}\norm{Az - b}^2\Big\} \nonumber\\
&\geq \min_{z \in z}\! \Big\{ f(\bar{z}(\lambda)) \!-\!
\iprod{A^T\lambda \!+\!\rho A^T\!(A\bar{z}(\lambda)\!-\!b),
z\!-\!\bar{z}(\lambda)}  + \iprod{\mu, A z-b} \!+\!
\frac{\rho}{2}\norm{Az-b}^2\Big\} \nonumber\\
& ~~~+ \min_{z \in Z}
\iprod{\nabla{\mathcal{L}}_{\rho}\left(\bar{z}(\lambda),\lambda\right),z-\bar{z}(\lambda)},
\end{align*}
where we use the convexity of $f$ and the properties of  minimum in
the first and the second inequality, respectively. Using now the
assumptions of the theorem and the definition of $\bar{d}_{\rho}$ we
obtain:
\begin{align*}
& d_{\rho}(\mu)\\
&\geq\! \min_{z \in Z}\!\Big\{ \!f(\bar{z}(\lambda)) \!\!-\!\!
\iprod{A^T\lambda \!+\! \rho A^T\! (A \bar{z}(\lambda) \!-\! b), z
\!-\! \bar{z}(\lambda)}
 \!+\! \iprod{\mu, A z \!-\! b} \!+\! \frac{\rho}{2}\norm{A z\!-\!b}^2\!\Big\}\!-\!C_Z \ei \\
& = \min_{z \in Z}\Big\{\!\iprod{A(z\!-\!\bar{z}(\lambda)), \mu
\!-\! \lambda} +
\frac{\rho}{2}\norm{A(z\!-\!\bar{z}(\lambda))}^2\!\Big\} \!+\!
\bar{d_{\rho}}(\lambda) \!+\! \iprod{\nabla{\bar{d}}_{\rho}(\lambda), \mu \!-\!\lambda} \!-\! C_Z \ei \\
&\geq \!\min_{z \in
\rset^n}\Big\{\!\iprod{A(z\!-\!\bar{z}(\lambda)), \mu \!-\!
\lambda} \!+\!
\frac{\rho}{2}\norm{A(z\!-\!\bar{z}(\lambda))}^2\!\Big\}\!+\!
\bar{d_{\rho}}(\lambda) \!+\!
\iprod{\nabla{\bar{d}}_{\rho}(\lambda), \!\mu \!-\!\lambda} \!-\!
C_Z \ei.
\end{align*}
By taking into account that $\min_{\xi  \in \rset^n}
{\frac{1}{2}\norm{\xi}^2 + g^T\xi}=-\frac{1}{2}\norm{g}^2$ and using
the definition of $C_Z$, we obtain from the last expression the
right-hand side inequality.
\end{proof}
%% End of the proof.

Note that the first inequality helps us construct a quadratic model
which bounds from below the function $d_{\rho}$, when the exact
values of the dual function and its gradients are unknown. The
second inequality can be  viewed as an approximation of the
concavity condition on $d_{\rho}$.   The two models, the linear and
the quadratic one,  use only approximate function values and
approximate gradients evaluated at certain points. A more general
framework for inexact gradient methods can be found in
\cite{DevGli:11}.

It can be easily  proved that under the assumptions of Theorem
\ref{th:cond_echiv}, the following relation holds between the true
and approximate gradient:
\begin{equation*}
\norm{\nabla{\bar{d}}_{\rho}(\lambda) - \nabla{d}_{\rho}(\lambda)}
\leq \sqrt{2L_{\mathrm{d}} \left(1+\sqrt{2\Lp}R_{\text{p}}\right)
\ei} ~~~\forall \lambda \in \rset^m.
\end{equation*}

\subsection{Inexact dual gradient method}\label{subsec_dual_gradient}
In this section we provide the convergence rate analysis  of an
inexact dual gradient ascent method. Let $\set{\alpha_j}_{j\geq
0}$ be a sequence of positive numbers. We denote by $S_k :=
\sum_{j=0}^k\alpha_j$. In this section we consider an inexact
gradient method for updating the dual variables:
\begin{equation}\label{iter_outer}
\boxed{~~~\lambda_{k+1} := \lambda_k + \alpha_k \nabla{\bar{d}_{\rho}}(\lambda_k), ~~~ \tag{\textbf{IDGM}}}
\end{equation}
where  $\alpha_k \in [\underline{L}^{-1},
L_{\mathrm{d}}^{-1}]\subset (0, +\infty)$ is a given step size and
$\underline{L} \geq L_{\mathrm{d}}$. We recall that the inexact
gradient is given by:
\begin{equation*}
\nabla{\bar{d}_{\rho}}(\lambda_k) := A \bar{z}_k - b,
\end{equation*}
where $\bar{z}_k := \bar{z}(\lambda_k)$.  The following theorem provides estimates on dual suboptimality for the scheme \eqref{iter_outer}.

%% Theorem 3.3.
\begin{theorem}\label{theorem_dual}
Suppose that the conditions of Theorem \ref{th:cond_echiv} are
satisfied. Let $\set{\lambda_k}_{k\geq 0}$ be a sequence generated
by \eqref{iter_outer} and $\{\hat{\lambda}_k\}_{k\geq 0}$ be an
average sequence derived from $\set{\lambda_k}_{k\geq 0}$ as
$\hat{\lambda}_k := S^{-1}_k\sum_{j=0}^k\alpha_j\lambda_{j+1}$.
Then, for any $k \geq 1$, the following estimate for the dual
suboptimality holds:
\begin{equation}\label{eq_conv}
f^{*} - d_{\rho}(\hat{\lambda}_k) \leq
\frac{\underline{L}R_\text{d}^2}{2(k+1)} +
\left(1+\sqrt{2\Lp}R_{\text{p}}\right)\ei,
\end{equation}
where we define  $R_\text{d} := \norm{\lambda_0 - \lambda^{*}}$.
\end{theorem}

%% Proof of Theorem 3.3.
\begin{proof}
Let $r_j := \norm{\lambda_j -\lambda^{*}}$. By using
\eqref{iter_outer} and the estimates \eqref{ineq_approx} we have:
\begin{align}\label{eq:estimate1}
r_{j+1}^2 &= r^2_j +2\iprod{\lambda_{j+1} -\lambda_j,\lambda_{j+1} -\lambda^{*}}  -\norm{\lambda_{j+1} - \lambda_j}^2 \nonumber\\
 &\!\!\!\!\!\!\!\overset{\tiny\eqref{iter_outer}}{=} r_j^2 - 2\alpha_j\iprod{\nabla{\bar{d}}_{\rho}(\lambda_j), \lambda^{*}-\lambda_j}  - \left(1
-\alpha_jL_{\mathrm{d}}\right)\norm{\lambda_{j+1}-\lambda_j}^2 \nonumber\\
& ~\qquad+
2\alpha_j\left[\iprod{\nabla{\bar{d}}_{\rho}(\lambda_j),\lambda_{j+1}
- \lambda_j} - \frac{L_{\mathrm{d}}}{2}\norm{\lambda_{j+1} -
\lambda_j}^2\right]
\nonumber\\
&\!\!\overset{\tiny\eqref{ineq_approx}}{\leq} r_j^2 +
2\alpha_j\left[\bar{d}_{\rho}(\lambda_j) -
d_{\rho}(\lambda^{*})\right] + 2\alpha_j\left[d_{\rho}(\lambda_{j
+ 1}) - \bar{d}_{\rho}(\lambda_j)\right]\nonumber\\
& ~~\quad+ 2\alpha_jC_Z\ei - (1 -\alpha_jL_{\mathrm{d}})\norm{\lambda_{j+1} -\lambda_j}^2 \nonumber\\
&\leq ~r_j^2 - 2\alpha_j\left[d_{\rho}(\lambda^{*}) -
d_{\rho}(\lambda_{j+1})\right] + 2\alpha_jC_Z\ei.
\end{align}
Here the last inequality follows from $\alpha_j\in
[\underline{L}^{-1},L_{\mathrm{d}}^{-1}]$. Summing up the last
inequality from $j =0$ to $k$ and taking into account that
$d_{\rho}(\lambda^{*}) \equiv f^{*}$, we obtain:
\begin{equation*}
\sum_{j=0}^{k}2\alpha_j[f^{*}-d_{\rho}(\lambda_{j+1})] \leq
R_\text{d}^2 + 2S_kC_Z\ei.
\end{equation*}
Now, by the concavity of $d_{\rho}$ and the definition of
$\hat{\lambda}_k$, this inequality implies:
\begin{equation*}
S_k\left[f^{*} - d_{\rho}(\hat{\lambda}_k)\right] \leq
\frac{R_\text{d}^2}{2} + S_kC_Z\ei.
\end{equation*}
Note that $f^{*} - d_{\rho}(\hat{\lambda}_k)\geq 0$ and $S_k \geq \underline{L}^{-1}(k+1)$. The last inequality together with the definition of $C_Z$ imply
\eqref{eq_conv}.
\end{proof}
%% End of the proof.

Next, we show how we can compute  an approximate  solution of the
primal problem \eqref{eq:primal_prob}. For this approximate solution
we estimate the feasibility violation  and the bound on the
suboptimality for \eqref{eq:primal_prob}. Let us consider  the
following average sequence:
\begin{equation}\label{averaging_gradient}
\hat{z}_k := S_k^{-1}\sum_{j=0}^k\alpha_j\bar{z}_{j}.
\end{equation}
Since $\bar{z}_j\in Z$ for  all $j\geq 0$ and $Z$ is convex, then
$\hat{z}_k\in Z$. From the iteration of Algorithm \eqref{iter_outer}
and \eqref{averaging_gradient}, by induction, we have:
\begin{equation}\label{new_form}
\lambda_{k+1} = \lambda_0 + S_k(A\hat{z}_k - b).
\end{equation}
The following two theorems provide bounds on the primal
infeasibility and the primal suboptimality, respectively.

%% Theorem 3.4.
\begin{theorem}\label{theorem_fesa}
Under assumptions of Theorem \ref{theorem_dual}, the sequence $\hat{z}_k$ generated by \eqref{averaging_gradient} satisfies the following
upper bound on the infeasibility for primal problem \eqref{eq:primal_prob}:
\begin{equation}\label{eq_fezab}
\norm{A\hat{z}_k - b} \leq \nu(k,\varepsilon_{\mathrm{in}}),
\end{equation}
where $\nu(k,\varepsilon_{\mathrm{in}}) :=
\frac{2\underline{L}R_\text{d}}{k+1} +
\sqrt{\frac{2\underline{L}\left(1+\sqrt{2\Lp}R_{\text{p}}\right)\varepsilon_{\mathrm{in}}}{k+1}}$.
\end{theorem}

%% Proof of Theorem 3.4.
\begin{proof}
From \eqref{eq:estimate1} we have:
\begin{equation*}
\norm{\lambda_{j+1} \!-\! \lambda^{*}}^2 \leq  \norm{\lambda_j
\!-\! \lambda^{*}}^2 - 2\alpha_j[d_{\rho}(\lambda^{*}) \!-\!
d_{\rho}(\lambda_{j+1})] + 2\alpha_jC_Z\varepsilon_{\mathrm{in}}
\leq \norm{\lambda_j \!-\! \lambda^{*}}^2 +
2\alpha_jC_Z\varepsilon_{\mathrm{in}}.
\end{equation*}
Here the last inequality follows from the fact that
$d_{\rho}(\lambda^{*}) - d_{\rho}(\lambda_{j+1})\geq 0$ and
$\alpha_j > 0$. By induction, it follows from the above inequality
that:
\begin{equation}\label{eq:est2}
\norm{\lambda_{k+1} - \lambda^{*}}^2  \leq
\norm{\lambda_0-\lambda^{*}}^2 + 2S_kC_Z\ei.
\end{equation}
Now, from \eqref{new_form} we have:
\begin{align*}
\norm{\lambda_{k+1} - \lambda^{*}}^2  &=
\norm{\lambda_0 - \lambda^{*} + S_k(A\hat{z}_k-b)}^2 \geq
\left[ \norm{\lambda_0-\lambda^{*}} - S_k\norm{A\hat{z}-b} \right]^2\\
& = \norm{\lambda_0-\lambda^{*}}^2 - 2S_k\norm{\lambda_0-\lambda^{*}}\norm{A\hat{z}_k
- b} + S_k^2\norm{A\hat{z}_k - b}^2.
\end{align*}
Substituting this inequality into \eqref{eq:est2} we obtain:
\begin{equation*}
S_k^2\norm{A\hat{z}_k \!-\! b}^2 - 2S_k\norm{\lambda_0 \!-\!
\lambda^{*}}\norm{A\hat{z}_k \!-\! b} \leq 2S_kC_Z\ei.
\end{equation*}
The last inequality implies:
\begin{equation*}
\norm{A\hat{z}_k - b} \leq \frac{R_\text{d} + [R_\text{d}^2 + 2S_kC_Z\ei]^{1/2}}{S_k} \leq \frac{2R_\text{d}}{S_k} +
\sqrt{\frac{2C_Z\varepsilon_{\mathrm{in}}}{S_k}}
\end{equation*}
Note that $S_k \geq \underline{L}^{-1}(k+1)$. This inequality implies \eqref{eq_fezab}.
\end{proof}
%% End of the proof.

%% Theorem 3.5.
\begin{theorem}\label{theorem_primal}
Under the assumptions of Theorem \ref{theorem_fesa}, the primal
suboptimality can  be characterized by the following lower and upper
bounds:
\begin{equation*}
 -\left[ \norm{\lambda^{*}}  + \frac{\rho}{2}\nu(k,\varepsilon_{\mathrm{in}})  \right]  \nu(k,\varepsilon_{\mathrm{in}}) \leq  f(\hat{z}_k) - f^{*}
\leq  \frac{\underline{L}\norm{\lambda_0}^2}{2(k + 1)} +
\left(1+\sqrt{2L_{\mathrm{p}}}R_{\mathrm{p}}\right)\varepsilon_{\mathrm{in}}.
\end{equation*}
\end{theorem}

%% Proof of Theorem 3.5.
\begin{proof}
Let us first prove the left-hand side inequality. Since
$f^{*}=d_{\rho}(\lambda^{*})$, by using the definition of
$\mathcal{L}_{\rho}(\hat{z}_k,\lambda^{*})$ and the Cauchy-Schwartz
inequality we get:
\begin{align*}
f^{*} &= d_{\rho}(\lambda^{*}) \leq
\mathcal{L}_{\rho}(\hat{z}_k,\lambda^{*})= f(\hat{z}_k) + \iprod{\lambda^{*}, A \hat{z}_k - b} + \frac{\rho}{2}\norm{A\hat{z}_k - b}^2\\
&\leq f(\hat{z}_k)+\norm{\lambda^{*}}\norm{A\hat{z}_k - b} +
\frac{\rho}{2}\norm{A \hat{z}_k-b}^2 \leq f(\hat{z}_k) +
\norm{\lambda^{*}}\nu(k,\varepsilon_{\mathrm{in}}) +
\frac{\rho}{2}\nu(k,\varepsilon_{\mathrm{in}})^2.
\end{align*}
Here, the last inequality follows from Theorem \ref{theorem_fesa}.
In order to prove the right-hand side inequality we first use the
convexity of $\mathcal{L}_{\rho}$ and the assumptions of Theorem
\ref{th:cond_echiv}:
\begin{align*}
\mathcal{L}_{\rho}(\bar{z}_j,\lambda_j) \leq
\mathcal{L}_{\rho}\left(z^*(\lambda_j),\lambda_j
\right)-\iprod{\nabla \mathcal{L}_{\rho}(\bar{z}_j,\lambda_j),
z^*(\lambda_j)-\bar{z}_j }\leq d_{\rho}(\lambda_j)+C_Z\ei.
\end{align*}
Previous inequality together with the definition of
$\mathcal{L}_{\rho}$ and $d_{\rho}(\lambda_j) \leq f^*$ lead to:
\begin{equation*}
f(\bar{z}_j)+\iprod{\lambda_j, A \bar{z}_j - b} +
\frac{\rho}{2}\norm{A\bar{z}_j - b}^2 - f^{*} \leq
C_Z\varepsilon_{\mathrm{in}}.
\end{equation*}
Using the iteration of Algorithm \eqref{iter_outer} and
$\alpha_j\leq \rho = L_{\mathrm{d}}^{-1}$ we obtain:
\begin{align*}
f(\bar{z}_j) - f^{*}  &\leq C_Z\varepsilon_{\mathrm{in}} -
\iprod{\lambda_j, \alpha_j^{-1}(\lambda_{j+1} - \lambda_j)}
-\frac{\rho\alpha_j^{-2}}{2}\norm{\lambda_{j+1}-\lambda_j}^2\\
&\leq \frac{1}{2\alpha_j}(\norm{\lambda_j}^2 -
\norm{\lambda_{j+1}}^2)+C_Z\varepsilon_{\mathrm{in}}.
\end{align*}
Multiplying this inequality with  $\alpha_j$ and then summing up
these inequalities from $j=0$ to $k$ we get:
\begin{equation*}
\sum_{j=0}^k\alpha_j(f(\bar{z}_j)-f^{*}) \leq
\frac{1}{2}\left(\norm{\lambda_0}^2-\norm{\lambda_{k+1}}^2\right)
+ S_kC_Z\varepsilon_{\mathrm{in}}\leq
\frac{1}{2}\norm{\lambda_0}^2 +  S_kC_Z\varepsilon_{\mathrm{in}}.
\end{equation*}
Now, using the definition of $\hat{z}_k$ and the convexity of $f$
we can deduce that:
\begin{equation*}
f(\hat{z}_k) - f^{*} \leq \frac{\norm{\lambda_0}^2}{2S_k} +
C_Z\varepsilon_{\mathrm{in}}.
\end{equation*}
Finally, by taking into account that $S_k \geq \underline{L}(k+1)$, we get from the last estimate the right-hand side inequality.
\end{proof}
%% End of the proof.

Let us fix the outer accuracy $\eo$. We want to find  the number
of outer iterations $k_{\mathrm{out}}$ and a relation between
$\eo$ and $\ei$ such that after this number of iterations of
Algorithm \eqref{iter_outer} the estimates
$(\hat{z}_{k_{\text{out}}},\hat{\lambda}_{k_{\text{out}}})$
satisfy \eqref{eq:out1}. For this purpose we can choose the
following values for $k_{\mathrm{out}}$ and $\ei$:
\begin{equation}
\label{eq_choice_IDGM} k_{\mathrm{out}} := \left\lfloor
\frac{\underline{L}R_\text{d}^2}{\eo}\right\rfloor ~\textrm{and}~
\ei := \frac{1}{2\left(1+\sqrt{2\Lp}R_{\text{p}}\right)}\eo.
\end{equation}
Thus, we  conclude from Theorems \ref{theorem_dual},
\ref{theorem_fesa} and \ref{theorem_primal} that for these choices
of $k_{\mathrm{out}}$ and $\ei$ the following estimates hold:

\begin{align*}
&f^{*}  - d_{\rho}(\hat{\lambda}_{k_{\mathrm{out}}})  \leq \eo, ~
\hat{z}_{k_{\text{out}}} \in Z,  ~\norm{A \hat{z}_{k_{\text{out}}} -b} \leq \frac{3}{R_\text{d}}\varepsilon_{\mathrm{out}} ~~\mathrm{and}\\
-&\left(\frac{3\norm{\lambda^*}}{R_\text{d}} +
\frac{9\rho}{2R_\text{d}^2}\eo\right)\eo \leq
f(\hat{z}_{k_{\text{out}}})  - f^{*} \leq \left(
 \frac{1}{2} + \frac{\norm{\lambda_0}^2}{2R_\text{d}^2}   \right
 )\eo.
\end{align*}

Finally, we are ready to summarize the above convergence rate analysis in the following algorithm.

%%% ALGORITHM 1. %%%
\begin{algorithm}{\textit{Inexact dual gradient method} ($\textbf{IDGM}$)}\label{alg:A1}
\newline
\textbf{Initialization:} Choose parameters $\rho > 0$ and $0 <
L_{\mathrm{d}} \leq \underline{L}$. Choose an initial point
$\lambda_0\in\rset^m$.
\newline
\textbf{Outer iteration: } For $k=0, 1,\dots, k_{\mathrm{out}}$, perform:
\begin{itemize}
\item[]Step 1. (\textit{Inner loop}). For given $\lambda_k$, solve
the inner problem \eqref{eq:inner_problem} with accuracy $\ei$,
such that one of the stopping criterions \eqref{eq:criterion_a} -
\eqref{eq:inner_criterion} are satisfied, to obtain $\bar{z}_k$.
\item[]Step 2. Form the approximate gradient vector of $d_{\rho}$
as $\nabla{\bar{d}}_{\rho}(\lambda_k) := A\bar{z}_k - b$.
\item[]Step 3. Select $\alpha_k\in [\underline{L}^{-1},
L_{\mathrm{d}}^{-1}]$ and update $S_{k} := \sum_{j=0}^k \alpha_j$.
\item[]Step 4. Update $\lambda_{k+1} := \lambda_k +
\alpha_k\nabla{\bar{d}}_{\rho}(\lambda_k)$.
\end{itemize}
\textbf{Output}: $\hat{z}_{k_{\mathrm{out}}} := S_{k_{\mathrm{out}}}^{-1}\sum_{j=0}^{k_{\mathrm{out}}}\alpha_j\bar{z}_j$.
\newline
\end{algorithm}
%%% END OF ALGORITHM 1 %%%

The penalty parameter $\rho$ in this algorithm  can   be also
updated  adaptively by using e.g. the procedure given in
\cite{Ham:05}.

%%%%%%%%%%%%%%%%%%%%%%%%%%%%%%%%%%%%%%%%%%%%%%%%%%%%%%%%%%%%%%%%%%%%%%%%%%%%%%%%
%%% III. B. Inexact dual fast gradient method
%%%%%%%%%%%%%%%%%%%%%%%%%%%%%%%%%%%%%%%%%%%%%%%%%%%%%%%%%%%%%%%%%%%%%%%%%%%%%%%%

\subsection{Inexact  dual fast gradient method}\label{subsec_dual_fast_gradient}
In this subsection we discuss a fast gradient scheme for solving
the augmented Lagrangian dual problem \eqref{eq:aug_dual_prob}.
Fast gradient schemes were first proposed by Nesterov
\cite{Nes:04a} and have also been discussed in the context of dual
decomposition in \cite{NecSuy:08}. A modification of these schemes
for the case of inexact information can be also found in
\cite{DevGli:11}. We shortly present such a scheme as follows.
Given a positive sequence $\set{\theta_k}_{k\geq 0}\subset (0,
+\infty)$ with $\theta_0 = 1$, we define $S_k :=
\sum_{j=0}^k\theta_j$. Let us assume that the sequence
$\set{\theta_k}_{k\geq 0}$ satisfies $\theta_{k+1}^2 = S_{k+1}$
for all $k\geq 0$. This condition leads to:
\begin{equation}\label{eq:theta_update}
\theta_{k+1} := \frac{1}{2}( 1 + \sqrt{4\theta_k^2 + 1}) ~~\forall
k\geq 0~~\textrm{and}~\theta_0 := 1.
\end{equation}
Note that the sequence $\set{\theta_k}_{k\geq 0}$ generated by
\eqref{eq:theta_update} is increasing and satisfies:
\begin{equation}\label{eq:cond_theta}
0.5(k+1) \leq \theta_k \leq k+1  ~~\forall k\geq 0.
\end{equation}
We can also obtain $0.25(k+1)(k+2) < S_k < 0.5(k+1)(k+2)$ and
$\sum_{j=0}^kS_j < (k+1)(k+2)(k+3)/3$. Now, we consider the dual
fast gradient scheme as follows: given an initial point
$\lambda_0\in \rset^m$, we define two sequences of the dual
variables $\set{\lambda_k}_{k \geq 0}$ and $\set{\mu_k}_{k \geq
0}$ as:
\begin{equation}\label{update_1}
\boxed{~~
\begin{cases}
&\mu_k ~~~:= \lambda_k + L_{\mathrm{d}}^{-1}\nabla{\bar{d}_{\rho}}(\lambda_k)\\
&\lambda_{k + 1} := \left(1 - a_{k + 1} \right)\mu_k  + a_{k + 1}\Big[\lambda_0  +
L_{\mathrm{d}}^{-1}\sum_{i=0}^k\theta_i\nabla{\bar{d}}_{\rho}(\lambda_i)\Big],
\end{cases}~~\tag{\textbf{IDFGM}}}
\end{equation}
where the sequence $a_{k+1} := S_{k+1}^{-1}\theta_{k+1}$.

The following lemma,  which  represents an extension of the results
in \cite{NecSuy:08,Nes:04a} to the inexact case (see also
\cite{DevGli:11}), will be used to derive estimates  on both primal
and dual suboptimality and also primal infeasibility for the
proposed method.

%%+ Lemma 3.6.
\begin{lemma}\cite{DevGli:11,NecSuy:08}\label{prop_rec}
Under the assumptions of Theorem \ref{th:cond_echiv}, the two
sequences $\set{(\lambda_k,\mu_k)}_{k\geq 0}$ generated by the
dual fast gradient scheme \eqref{update_1} satisfy:
\begin{eqnarray}\label{eq_rec}
S_kd_{\rho}(\mu_k) &&\geq \max_{\lambda\in\rset^m}
\left\{\sum_{j=0}^k\theta_j\big[\bar{d}_{\rho}(\lambda_j) +
\iprod{\nabla{\bar{d}}_{\rho}(\lambda_j),
\lambda-\lambda_j}\big]  - \frac{L_{\mathrm{d}}}{2}\norm{\lambda-\lambda_0}^2 \right\} \nonumber\\
[-1.8ex]\\[-1.8ex]
&& ~~~-
\left(1+\sqrt{2\Lp}R_{\text{p}}\right)\varepsilon_{\mathrm{in}}\sum_{j=0}^kS_j
~~ \forall k\geq 0.\nonumber
\end{eqnarray}
\end{lemma}

The next theorem gives an estimate on  dual suboptimality.

%% Theorem 3.7.
\begin{theorem}\label{theorem_dual_optim}
Under the assumptions of Theorem \ref{th:cond_echiv},  let
$\set{(\lambda_k,\mu_k)}_{k\geq 0}$ be the two sequences generated
by the scheme \eqref{update_1}. Then, an estimate on dual
suboptimality is given by the following expression:
\begin{equation*}
f^{*}-d_{\rho}(\mu_k)\leq\!
\frac{2L_{\mathrm{d}}R_\text{d}^2}{(k+1)(k+2)}+\frac{4(k+3)}{3}\left(1\!+\!\sqrt{2\Lp}R_{\text{p}}\right)\ei.
\end{equation*}
\end{theorem}

%% Proof of Theorem 3.7.
\begin{proof}
By using  inequality  \eqref{ineq_approx} in \eqref{eq_rec} we
obtain:
\begin{align*}
S_kd_{\rho}(\mu_k) &\geq S_kd_{\rho}(\lambda^{*}) - \frac{L_{\mathrm{d}}}{2}\norm{\lambda^{*}-\lambda_0}^2 - C_Z\varepsilon_{\mathrm{in}}\sum_{j=0}^kS_j.
\end{align*}
Now, using the fact that $S_k > 0.25(k+1)(k+2)$ and
$\sum_{j=0}^kS_j < (k+1)(k+2)(k+3)/3$ and the definition of
$C_Z=1+\sqrt{2\Lp}R_{\text{p}}$, we obtain our result.
\end{proof}
%% End of the proof.

We further define the following primal average sequence:
\begin{equation}\label{primal_point}
\hat{z}_k := S_k^{-1}\sum_{j=0}^k\theta_j\bar{z}_i.
\end{equation}
Next theorem gives an estimate  on  infeasibility of $\hat{z}_k$
for the original problem \eqref{eq:primal_prob}.

%% Theorem 3.8.
\begin{theorem}\label{theorem_primal_fesa}
Under the conditions of Theorem \ref{theorem_dual_optim}, the
point $\hat{z}_k$ defined by \eqref{primal_point} satisfies the
following estimate on primal feasibility violation:
\begin{equation}\label{eq:primal_fesa}
\norm{A\hat{z}_k - b} \leq v(k,\varepsilon_{\textrm{in}}),
\end{equation}
where $v(k,\ei) :=\frac{8L_{\mathrm{d}}R_\text{d}}{(k+1)(k+2)}+
4\sqrt{\frac{2L_{\mathrm{d}}(k+3)\left(1+\sqrt{2\Lp}R_{\text{p}}\right)\ei}{3(k+1)(k+2)}}$.
\end{theorem}

%% Proof of Theorem 3.8.
\begin{proof}
By the definition of $\bar{d}_{\rho}$ and
$\nabla{\bar{d}}_{\rho}$, the convexity of $f$ and
$\norm{\cdot}^2$, and inequality \eqref{primal_point} we have:
\begin{align*}
\sum_{j=0}^k\theta_j \!\left[ \bar{d}_{\rho}(\lambda_j) +
\iprod{\nabla{\bar{d}}_{\rho}(\lambda_j),\lambda \!-\!  \lambda_j}
\right] &=
\!\sum_{j=0}^k\theta_jf(\bar{z}_j) + S_k\iprod{\lambda, A\hat{z}_k \!-\! b} \!+\! \sum_{j=0}^k\theta_j\frac{\rho}{2}\norm{A\bar{z}_j-b}^2 \nonumber\\
&\geq S_kf(\hat{z}_k) + S_k\iprod{\lambda, A\hat{z}_k - b)} +
\frac{S_k}{2L_{\mathrm{d}}}\norm{A\hat{z}_k - b}^2.
\end{align*}
Substituting this inequality into \eqref{eq_rec} we obtain:
\begin{align}\label{eq:proof10}
d_{\rho}(\mu_k) &\geq  f(\hat{z}_k)\! +\! \max_{\lambda\in\rset^m}\Big\{\! \iprod{\lambda, A\hat{z}_k - b)} -
\frac{L_{\mathrm{d}}}{2S_k}\norm{\lambda\!-\!\lambda_0}^2\Big\}  \nonumber\\
& ~~~+ \frac{\rho}{2}\norm{A\hat{z}_k - b}^2 - C_Z\ei
S_k^{-1}\sum_{j=0}^kS_j.
\end{align}
On the one hand, we can write:
\begin{align}\label{eq:proof11}
d_{\rho}(\mu_k) - f(\hat{z}_k) - \frac{\rho}{2}\norm{A\hat{z}_k - b}^2 &\leq d_{\rho}(\lambda^{*}) - f(\hat{z}_k) - \frac{\rho}{2}\norm{A\hat{z}_k - b}^2
\nonumber\\
& = \min_{z \in Z}\mathcal{L}_{\rho}(z,\lambda^{*}) \!-\!
f(\hat{z}_k) \!-\! \frac{\rho}{2}\norm{A\hat{z}_k\!-\!b}^2 \leq
\iprod{\lambda^{*}, A \hat{z}_k\!-\!b}.
\end{align}
On the other hand, we have:
\begin{align}\label{eq:proof12}
\max_{\lambda\in\rset^m}&\Big\{ -\frac{L_{\mathrm{d}}}{2S_k}\norm{\lambda-\lambda_0}^2 + \iprod{\lambda, A \hat{z}_k-b} \Big\} =
\frac{S_k}{2L_{\mathrm{d}}}\norm{A\hat{z}_k-b}^2 +
\iprod{\lambda_0, A \hat{z}_k-b}.
\end{align}
Substituting \eqref{eq:proof11} and \eqref{eq:proof12} into
\eqref{eq:proof10} we obtain:
\begin{equation*}
\frac{S_k}{2L_{\mathrm{d}}}\norm{A\hat{z}_k-b}^2 +
\iprod{\lambda_0 - \lambda^{*}, A \hat{z}_k-b} \leq  C_Z\ei
S_k^{-1}\sum_{j=0}^kS_j.
\end{equation*}
If we define $\xi := \norm{A\hat{z}_k-b}$, then the last inequality implies $\frac{(k+1)(k+2)}{8L_{\mathrm{d}}}\xi^2 - R_\text{d}\xi
\leq \frac{4(k+3)}{3}C_Z\ei$. Therefore, we obtain $\xi \leq \nu(k,\ei)$, where $\nu(\cdot,\cdot)$ is defined in \eqref{eq:primal_fesa}.
\end{proof}
%% End of the proof.

Finally, we characterize the primal suboptimality for optimization
problem \eqref{eq:primal_prob}.

%% Theorem 3.9.
\begin{theorem}\label{theorem_primal_optim}
Under the conditions of Theorem \ref{theorem_primal_fesa}, the
following  estimates hold on primal suboptimality:
\begin{align*}
%\label{bound_primal_optim}
-&\Big[ \norm{\lambda^{*}}  \!+\!
\frac{\rho}{2}\nu(k,\varepsilon_{\mathrm{in}}) \Big]
\nu(k,\varepsilon_{\mathrm{in}}) \leq f(\mathbf{ \hat
z}_k)\!-\!f^{*} \!\leq
\frac{2L_{\mathrm{d}}\norm{\lambda_0}^2}{(k\!+\!1)(k\!+\!2)}
  +\frac{4(k\!+\!3)}{3}\left(\!1\!+\!\sqrt{2\Lp}R_{\text{p}}\!\right)\ei.
\end{align*}
\end{theorem}

%% Proof of Theorem 3.9.
\begin{proof}
The left-hand side inequality can be obtained similarly as in
Theorem \ref{theorem_primal}. We now prove the right-hand side. From
\eqref{eq:proof10} and \eqref{eq:proof12} we have:
\begin{align*}
d_{\rho}(\mu_k) &\geq f(\hat{z}_k) + \frac{S_k}{2L_{\mathrm{d}}}\norm{A\hat{z}_k-b}^2 + \iprod{\lambda_0, A \hat{z}_k-b} + \frac{\rho}{2}\norm{A\hat{z}_k - b}^2
- C_Z\ei S_k^{-1}\sum_{j=0}^kS_j \nonumber\\
&\geq f(\hat{z}_k) - \frac{2L_{\mathrm{d}}}{(k+1)(k+2)}\norm{\lambda_0}^2 - \frac{4(k+3)}{3}C_Z\ei.
\end{align*}
Therefore, we get:
\begin{align*}
f(\mathbf{ \hat z}_k) - d_{\rho}(\mu_k) \leq
\frac{2L_{\mathrm{d}}}{(k+1)(k+2)}\norm{\lambda_0}^2 +
\frac{4(k+3)}{3}C_Z\ei.
\end{align*}
Since $d_{\rho}(\mu_k) \leq f^{*}$, we obtain the right-hand side
inequality  from the last relation.
\end{proof}
%% End of the proof.

Similar to the previous subsection,   we assume that we fix the
outer accuracy $\eo$ and the goal is to find $k_{\mathrm{out}}$ and
a relation between $\eo$ and $\ei$ such that after
$k_{\mathrm{out}}$ outer iterations of the scheme \eqref{update_1}
relations \eqref{eq:out1} holds. We can take e.g.:
\begin{equation}
\label{eq_choice_IDFGM} k_{\mathrm{out}} := \left\lfloor
2R_\text{d}\sqrt{\frac{L_{\mathrm{d}}}{\eo}}\right\rfloor ~
\text{and} ~ \ei :=
\frac{3}{8\left(1\!+\!\sqrt{2\Lp}R_{\text{p}}\right)\!(k_{\mathrm{out}}\!+\!3)}\eo.
\end{equation}
Using now Theorems \ref{theorem_dual_optim},
\ref{theorem_primal_fesa} and \ref{theorem_primal_optim} we can
conclude that the following bounds for dual  suboptimality, primal
infeasibility and primal suboptimality hold:
\begin{align*}
&~~~f^{*}  - d_{\rho}(\hat{\lambda}_{k_{\mathrm{out}}}) \leq \eo,~
\hat{z}_{k_{\text{out}}} \in Z,
~\norm{A \hat{z}_{k_{\text{out}}} -  b} \leq \frac{3}{R_\text{d}}\varepsilon_{\mathrm{out}} ~~\mathrm{and}\\
&-\left(\frac{3\norm{\lambda^*}}{R_\text{d}} +
\frac{9\rho}{2R_\text{d}^2}\eo\right)\eo \leq
f(\hat{z}_{k_{\text{out}}}) - f^{*} \leq
\left(\frac{\norm{\lambda_0}^2 +
R_\text{d}^2}{2R_\text{d}^2}\right)\eo.
\end{align*}
We can summarize the above convergence rate analysis into the
following algorithm.
%%% ALGORITHM 2. %%%
\begin{algorithm}{\textit{Inexact dual fast gradient method} (\textbf{IDFGM})}\label{alg:A2}
\newline
\textbf{Initialization:} Choose parameters $\rho > 0$ and $\theta_0 := 1$. Choose an initial point $\lambda_0\in\rset^m$ and set $S_0 := 1$.
\newline
\textbf{Outer iteration: } For $k=0, 1,\dots, k_{\mathrm{out}}$, perform:
\begin{itemize}
\item[]Step 1. (\textit{Inner loop}). For given $\lambda_k$, solve
the inner problem \eqref{eq:inner_problem} with accuracy $\ei$,
such that one of the stopping criterions \eqref{eq:criterion_a} -
\eqref{eq:inner_criterion} are satisfied, to obtain $\bar{z}_k$.
\item[]Step 2. Form the approximate gradient vector of $d_{\rho}$
as $\nabla{\bar{d}}_{\rho}(\lambda_k) := A\bar{z}_k - b$.
\item[]Step 3. Update $\mu_k := \lambda_k +
L_{\mathrm{d}}^{-1}\nabla{\bar{d}}_{\rho}(\lambda_k)$. \item[]Step
4. Update $\theta_{k+1} := 0.5 \left(1 +
\sqrt{1+4\theta_k^2}\right)$,
 $S_{k+1} := S_k + \theta_{k+1}$ and $a_{k+1} := S_{k+1}^{-1}\theta_{k+1}$.
\item[]Step 4. Update $\lambda_{k+1} := (1-a_{k+1})\mu_k + a_{k+1} \left[
\lambda_0 +
L_{\mathrm{d}}^{-1}\sum_{j=0}^k\theta_j\nabla{\bar{d}}_{\rho}(\lambda_j)
\right]$.
\end{itemize}
\textbf{Output}: $\hat{z}_{k_{\mathrm{out}}} := S_{k_{\mathrm{out}}}^{-1}\sum_{j=0}^{k_{\mathrm{out}}}\theta_j\bar{z}_j$.
\newline
\end{algorithm}
%%% END OF ALGORITHM 2 %%%

As in previous section, the penalty parameter $\rho$ can  be also
updated adaptively by using the same  procedure as before.

%%%%%%%%%%%%%%%%%%%%%%%%%%%%%%%%%%%%%%%%%%%%%%%%%%%%%%%%%%%%%%%%%%%%%%%%%%%%%%%%%%%%%%%%%%%%%
\section{Complexity certification for linear MPC problems}\label{sec_MPC}

In this section we discuss different implementation aspects for the
application of the  algorithms derived  in Sections
\ref{subsec_dual_gradient} and \ref{subsec_dual_fast_gradient} in
the context of state-input constrained MPC for fast linear embedded
systems. We first prove that for linear MPC with quadratic stage and
final costs, the augmented Lagrangian function becomes strongly
convex and therefore the inner problems \eqref{eq:inner_problem} can
be solved in linear  time with a fast gradient scheme \cite{Nes:04}.
We also discuss how the different parameters, which appear in our
derived complexity bounds of Algorithms \eqref{iter_outer} and
\eqref{update_1}, can be computed such that tight estimates on the
total number of iterations can be derived and thus to facilitate the
implementation on linear embedded systems with state-input
constraints.

\subsection{Implementation aspects for MPC problems}\label{subsec_implement}
 We denote by $X_N$ a subset of the region of attraction for the
MPC scheme discussed in Section \ref{subsec_motivation}. A detailed
discussion on the stability of suboptimal MPC schemes can be found
e.g. in \cite{ScoMay:99}. For a given $x \in X_N$, we denote with
$z^{*}(x)$ an optimal solution for $(\textbf{P}(x))$ and with
$\lambda^{*}(x)$ an associated optimal multiplier. Usually, in MPC
problems the stage and final costs are  quadratic functions of the
form:
\begin{equation*}
\ell(x_i,u_i) := x_i^T Q x_i + u_i^T R u_i ~~ \text{and}
~~\ell_{\mathrm{f}}(x_N) := x_N^T P x_N,
\end{equation*}
where the matrices $Q$ and $P$ are positive semidefinite and $R$ is
positive definite. Note that in our formulation we do not require
strongly convex stage cost, i.e. we do not impose the matrices $Q$
and $P$ to be positive definite. The following lemma characterizes
the convexity properties of the augmented Lagrangian function.
%% Lemma 4.1.
\begin{lemma}\label{le:strong_covexity_of_MPC}
If the optimization problem \eqref{eq:Px} comes from a linear MPC
problem with quadratic stage and final costs, then the augmented
Lagrangian $\mathcal{L}_{\rho}(z,\lambda,x)$ is a strongly convex
quadratic function w.r.t. variable $z$.
\end{lemma}

%% Proof of Lemma 4.1.
\begin{proof}
If we consider quadratic  costs in the MPC problem \eqref{eq:LMPC},
then the objective function $f$ is quadratic, i.e. $f(z) :=
\frac{1}{2} z^T H z$,  where the Hessian  $H := \text{diag}(\tilde
Q, \tilde R)$ is positive semidefinite, with $\tilde Q :=
\text{diag}(Q, \cdots, Q, P)$ and $\tilde R := \text{diag}(R,
\cdots, R)$. Note that $\tilde R$ is positive definite, since we
assume $R$ to be positive definite. Using these notations, we can
rewrite the augmented Lagrangian in the form (see Section
\ref{subsec_motivation}):
\begin{align*}
\mathcal{L}_{\rho}(z,\lambda,x) := \frac{1}{2}z^T(H + \rho A^T
A)z+ (A^T \lambda-\rho A^Tb(x))z - b(x)^T \lambda +
\frac{\rho}{2}b(x)^Tb(x).
\end{align*}
It is straightforward to see that since $H$ is positive
semidefinite, then $z^T(H + \rho A^T A)z > 0$ for all $z$ which
satisfy $Az \neq 0$. On the other hand, if we consider the following
set $\set{z \in \rset^n |~ Az = 0}$,   which comes from the linear
dynamics, we can rewrite equivalently this set as
$\set{z \in \rset^n~|~z = \left[\begin{array}{c} \tilde{A} u\\ u\\
\end{array}\right], u \in \prod_{i=1}^N U}$, where $u:=\left[u_0^T
\cdots u_{N-1}^T\right]^T$ and the matrix $\tilde{A}$ is obtained
from the matrices $A_x$ and $B_u$ describing the dynamics of the
system. Further, since $Az=0$, we can write $z^T(H + \rho A^T A)z =
z^T H z =  u^T \tilde{A}^T  \tilde Q \tilde{A} u + u^T \tilde R u
>0$ for all $u \neq 0$. The last inequality follows from the fact
that $\tilde R$ is positive definite. In conclusion, we proved that
$H + \rho A^T A$ is a positive definite matrix  and therefore
$\mathcal{L}_{\rho}(z,\lambda,x)$ is a quadratic strongly convex
function in $z$.
\end{proof}

The previous lemma shows that  in the linear MPC case with quadratic
costs the objective function of the inner subproblems
$\mathcal{L}_{\rho}$ are quadratic strongly convex in the first
variable $z$. Moreover,  $\mathcal{L}_{\rho}$ has also Lipschitz
continuous gradient. Note that the convexity parameter
$\sigma_{\mathrm{p}}$ of this function can be computed easily:
\begin{equation*}
\sigma_{\mathrm{p}} := \lambda_{\min}(H+\rho A^TA),
\end{equation*}
and the Lipschitz constant $L_{\mathrm{p}}$ of the gradient of
$\mathcal{L}_{\rho}$ is:
\begin{equation*}
L_{\mathrm{p}} := \lambda_{\max}(H +\rho A^TA).
\end{equation*}
Note that since $\mathcal{L}_{\rho}(z,\lambda,x)$ is strongly convex
and with Lipschitz continuous gradient in the variable $z$, by
solving the inner problem \eqref{eq:inner_problem} with a fast
gradient scheme we can ensure stopping criterion
\eqref{eq:criterion_a} in a linear number of inner iterations
\cite{Nes:04}. Since the estimate for the number of inner iterations
depends on $\sigma_{\mathrm{p}}$, $L_{\mathrm{p}}$ and also on the
diameter $R_{\text{p}}$ of the set $Z$, we can see immediately that
this diameter can be computed easily for cases when the set $Z$ has
a specific structure. Note that the set $Z$ is a Cartesian product
and thus:
\begin{equation*}
R_{\text{p}} := \sqrt{(N-1)D_x^2+D_{x_{\mathrm{f}}}^2+N D_u^2},
\end{equation*}
where $D_x$, $D_{x_{\mathrm{f}}}$ and $D_u$ denotes the diameters of
$X$, $X_{\mathrm{f}}$ and  $U$, respectively. These diameters can be
computed explicitly for constraints sets defined e.g. by  boxes or
Euclidean balls,  which typically appear in the context of MPC
problems.
%We consider the situations when the state and input
%constraints are described by balls defined by some norm. For the
%simplicity of the exposition we will consider only the case of the
%set $X$. If the set $X$ is a box $X := \set{ x \in \rset^{n_x} ~|~
%l_x \leq x \leq u_x }$, then we have $D_x := \norm{u_x - l_x}$.
%For an Euclidean ball with center $0$ and radius $r$, i.e. $X
%:=\set{x \in \rset^{n_x} ~|~ \norm{x} \leq r }$ we obtain the
%diameter $D_x=2r$.

 Further, the estimates for the number of outer iterations depend
on the norm of the dual optimal solution.  We now discuss how we can
bound $\norm{\lambda^{*}}$ in the MPC case. We make use of the
result from \cite{DevGli:12}:
%% Lemma 4.2.
\begin{lemma}\cite{DevGli:12}
For the family of MPC problems $(\textbf{P}(x))_{x \in X_N}$ we
assume that there exists $r > 0$ such that $B(0,r) \subseteq \set{A
z - b(x) ~|~ z \in Z, x \in X_N}$, where $B(0,r)$ denotes the
Euclidean ball in $\rset^{N(n_x+n_u)}$ with center $0$ and radius
$r$. Then, the following upper bound on the norm of the dual optimal
solutions of MPC problems $(\textbf{P}(x))$ holds:
\begin{equation*}
\norm{\lambda^{*}(x)} \leq \frac{\max_{z \in Z } \iprod{H z^{*}(x),
z-z^{*}(x)}}{\bar{r}} ~~\forall x \in X_N,
\end{equation*}
where $\bar{r} := \max \big\{r ~|~ B(0, r) \subseteq \set{A Z -
b(x) ~|~ x \in X_N }\big\}$.
\end{lemma}

Based on the previous lemma, in \cite{RicMor:11}  upper bounds are
derived on  $\norm{\lambda^{*} (x)}$ for all $x \in X_N$  for linear
MPC problems  with  $X$, $X_{\mathrm{f}}$, $U$ and $X_N$ polyhedral
sets:
\begin{equation}\label{eq:bound_eth}
\mathcal{R}_\text{d} \geq \max_{x\in X_N}\norm{\lambda^*(x)}.
\end{equation}
Recall that  Lipschitz constant of the gradient of augmented dual
function is  $L_{\mathrm{d}}=1/\rho$.

%%%%%%%%%%%%%%%%%%%%%%%%%%%%%%%%%%%%%%%%%%%%%%%%%%%%%%%%%%%%%%%%%%%%%%%%%%%
%%% 4.2. Total complexity of solving MPC problems.
%%%%%%%%%%%%%%%%%%%%%%%%%%%%%%%%%%%%%%%%%%%%%%%%%%%%%%%%%%%%%%%%%%%%%%%%%%%
\subsection{Total complexity of solving MPC problems}\label{subsec_total}
Now, we assume that we know the outer accuracy $\eo$ and we want to
estimate the total number of iterations and also the number of flops
per inner and outer iterations which have to be performed by
Algorithms \eqref{iter_outer} or \eqref{update_1} in order to solve
the MPC problem $(\textbf{P}(x))$. For both algorithms we assume the
initialization $\lambda_0=0$ and the inner problems are solved using
the stopping criterion \eqref{eq:criterion_a}.

First, we  discuss the complexity certificates  in the case when
problem \eqref{eq:Px} is solved using Algorithm \eqref{iter_outer}
for all $x \in X_N$. We denote by $k_{\mathrm{in}}^G$ the number of
inner iterations which has to be performed in order to solve each
inner problem and by $k_{\mathrm{out}}^G$ the number of outer
iterations. From the discussion in Section
\ref{subsec_dual_gradient} an upper bound on the number of outer
iterations is given by:
\begin{equation}
\label{eq_koutgrad} k_{\mathrm{out}}^G:= \left\lfloor
\frac{L_{\mathrm{d}}\mathcal{R}_{\text{d}}^2}{\eo}\right\rfloor.
\end{equation}
Since we have proved that in the MPC case
$\mathcal{L}_{\rho}(\cdot,\lambda,x)$ is strongly convex with
convexity parameter $\sigma_{\text{p}}$ and has also Lipschitz
continuous gradient with constant $L_{\text{p}}$, in order to find a
point $\bar z_{k_{\mathrm{in}}^G}(\lambda)$ such that
$\mathcal{L}_{\rho}(\bar
z_{k_{\mathrm{in}}^G}(\lambda),\lambda,x)-\mathcal{L}_{\rho}(z^*(\lambda),\lambda,x)\leq
\ei^2$ we can apply a fast gradient scheme. From Theorem 2.2.3 in
\cite{Nes:04} and taking into account that
$\ei=\frac{1}{2\left(1+\sqrt{L_{\mathrm{p}}}R_\mathrm{p}\right)}\eo$
(see \eqref{eq_choice_IDGM}) we get that the number of inner
iterations for finding such a point does not exceed:
\begin{equation}
\label{eq_kingrad} k_{\mathrm{in}}^G := \left\lfloor
2\sqrt{\frac{L_{\mathrm{p}}}{\sigma_{\mathrm{p}}}}\ln \left(
\frac{3\sqrt{L_{\text{p}}}R_{\text{p}}\left(1+\sqrt{2\Lp}R_{\text{p}}\right)}{\eo}
\right) \right\rfloor.
\end{equation}

In the case of Algorithm \eqref{update_1}, the number of outer
iterations, according to the discussion in Section
\ref{subsec_dual_fast_gradient}, is given by:
\begin{equation}
\label{eq_koutfastgrad} k_{\mathrm{out}}^{FG} := \left\lfloor 2
\mathcal{R}_{\text{d}}\sqrt{\frac{L_{\text{d}}}{\eo}}\right\rfloor.
\end{equation}
Taking into account that in this case we consider that the inner
accuracy is chosen as
$\ei=\frac{3}{8\left(1+\sqrt{L_{\mathrm{p}}}R_\mathrm{p}\right)
\left(k_{\mathrm{out}}^{FG}+3\right)}\eo$ (see
\eqref{eq_choice_IDFGM}), then the number of inner iterations for
solving each inner problem will be given by:
\begin{equation}
\label{eq_kinfastgrad} k_{\mathrm{in}}^{FG} := \left\lfloor
2\sqrt{\frac{L_{\mathrm{p}}}{\sigma_{\mathrm{p}}}}\ln \left(
\frac{5\sqrt{L_{\text{d}}}\mathcal{R}_d\sqrt{L_{\mathrm{p}}}R_{\text{p}}\left(1+\sqrt{2\Lp}R_{\text{p}}\right)}{\eo\sqrt{\eo}}\right)\right\rfloor.
\end{equation}
%Thus, to conclude the previous discussion, we have that the total number of iterations which have to be performed in order to solve the MPC problems over a
%simulation horizon of length $N^{\mathrm{s}}$  is given by
%\begin{equation*}
%N^{\mathrm{tot}} := N^{\mathrm{s}} \cdot N^{\mathrm{out}}_i \cdot N^{\mathrm{in}}_i ~~\forall i=1,2.
%\end{equation*}
Further, we are also interested in finding the total number of flops
for both outer and inner iterations. For solving the inner problem
we use a simple fast gradient scheme for smooth strongly convex
objective functions, see e.g. \cite{Nes:04}. For this scheme, an
inner iteration will require $n^{\mathrm{flops}}_{\mathrm{in}} :=
N\left(3n_x^2+2n_xn_u+2n_u^2+10n_x+8n_u\right)$ flops. Regarding the
number of flops required by an outer iteration, the following values
can be established: $n^{\mathrm{flops,G}}_{\mathrm{out}} :=
N\left(2n_x^2+2n_xn_u+5n_x\right)+k_{\mathrm{in}}^{G}n^{\mathrm{flops}}_{\mathrm{in}}$
for Algorithm \eqref{iter_outer} and
$n^{\mathrm{flops,FG}}_{\mathrm{out}} :=
N\left(2n_x^2+2n_xn_u+10n_x\right)+k_{\mathrm{in}}^{FG}n^{\mathrm{flops}}_{\mathrm{in}}$
for Algorithm \eqref{update_1}, respectively.

%%%%%%%%%%%%%%%%%%%%%%%%%%%%%%%%%%%%%%%%%%%%%%%%%%%%%%%%%%%%%%%%%%%%%%%%%%%%%%%%%%%%%%%%%%%%%%%%%%%%%%%%%%%%%%
%%% 5. Numerical experiments.

\section{Numerical experiments} \label{sec_numerical}
In order to certify  the efficiency of the proposed algorithms, we
consider different numerical scenarios. We first analyze the
behavior of Algorithms \eqref{iter_outer} and \eqref{update_1} in
terms of CPU time and number of iterations for some practical MPC
problems and then we compare the CPU time, of our algorithms and of
other well known QP solvers used in the context of MPC, on randomly
generated QP problems. All the simulations were performed on a
Laptop with CPU Intel T6670 with 2.2GHz and 4GB RAM memory, using
Matlab R2008b. In all simulations we consider $\lambda_0=0$.

%%%%%%%%%%%%%%%%%%%%%%%%%%%%%%%%%%%%%%%%%%%%%%
%%%%%%%%%%%%%%%%%%%%%%%%%%%%%%%%%%%%%%%%%%%%%%%%%%%%%%%%%%%%%%%%
%%% 5.1. Practical MPC problems.
\subsection{Practical MPC problems}\label{subsec_simulations}
 In this section we apply  the newly developed  Algorithms
\eqref{iter_outer} and \eqref{update_1} on MPC problems for some
practical applications, i.e. a ball on plate system and an
oscillating masses system.

%% 5.1.1. Ball on plate system.
\subsubsection{Ball on plate system}
The first application discussed in this section is the ball on plate
system described in \cite{RicMor:11}. We consider   box constraints
for states $X$ and $X_\text{f}$, inputs $U$ and for the region of
attraction $X_N$ as in \cite{RicMor:11}, while for the stage costs
we take the matrices $Q = q_1q_1^T$, where $q_1 = [2 \;  1]^T$, $R =
1$ and we compute the terminal matrix $P$ as the solution of  the
LQR problem.

For different prediction horizons ranging  from $N = 5$ to $N=20$,
we analyze first the behavior of Algorithms \eqref{iter_outer} and
\eqref{update_1} in terms of the number of outer iterations. For
each prediction horizon length, we consider two different estimates
for the number of outer iterations depending on the way we compute
the upper bound on the optimal Lagrange multipliers $\lambda^*(x)$.
For Algorithm \eqref{iter_outer},  $k_{\mathrm{out}}^G$ is the
theoretical number of iterations  obtained using relation
\eqref{eq_koutgrad} with $\mathcal{R}_\mathrm{d}$ computed according
to \cite{RicMor:11} (see \eqref{eq:bound_eth}), while
$k_{\mathrm{out,samp}}^G$ is the average  number of iterations
obtained using our derived bound \eqref{eq_choice_IDGM} with
$R_\mathrm{d}$ computed exactly using \texttt{Gurobi 5.0.1} solver,
iterations which correspond to $50$ random initial states $x \in
X_N$. We also compute the average number of outer iterations
$k_{\mathrm{out,real}}^G$ observed in practice, obtained by imposing
stopping criteria $\abs{f(\hat z_{k_{\mathrm{out,real}}^G}) -
f^{*}}$ and $\norm{A \hat z_{k_{\mathrm{out,real}}^G} - b}$  less
than accuracy $\eo = 10^{-3}$. For Algorithm \eqref{update_1} we
compute in a similar way $k_{\mathrm{out}}^{FG}$ using
\eqref{eq_koutfastgrad}, $k_{\mathrm{out,samp}}^{FG}$ using
\eqref{eq_choice_IDFGM} and $k_{\mathrm{out,real}}^{FG}$ observed in
practice. In all simulations we  take  $\rho = 1$. The results for
both algorithms are reported in Figure
\ref{fig:outer_iterations_ball}.
\begin{figure}[h!]
\vskip-0.3cm
\centering\includegraphics[angle=0,height=4cm,width=12.5cm]{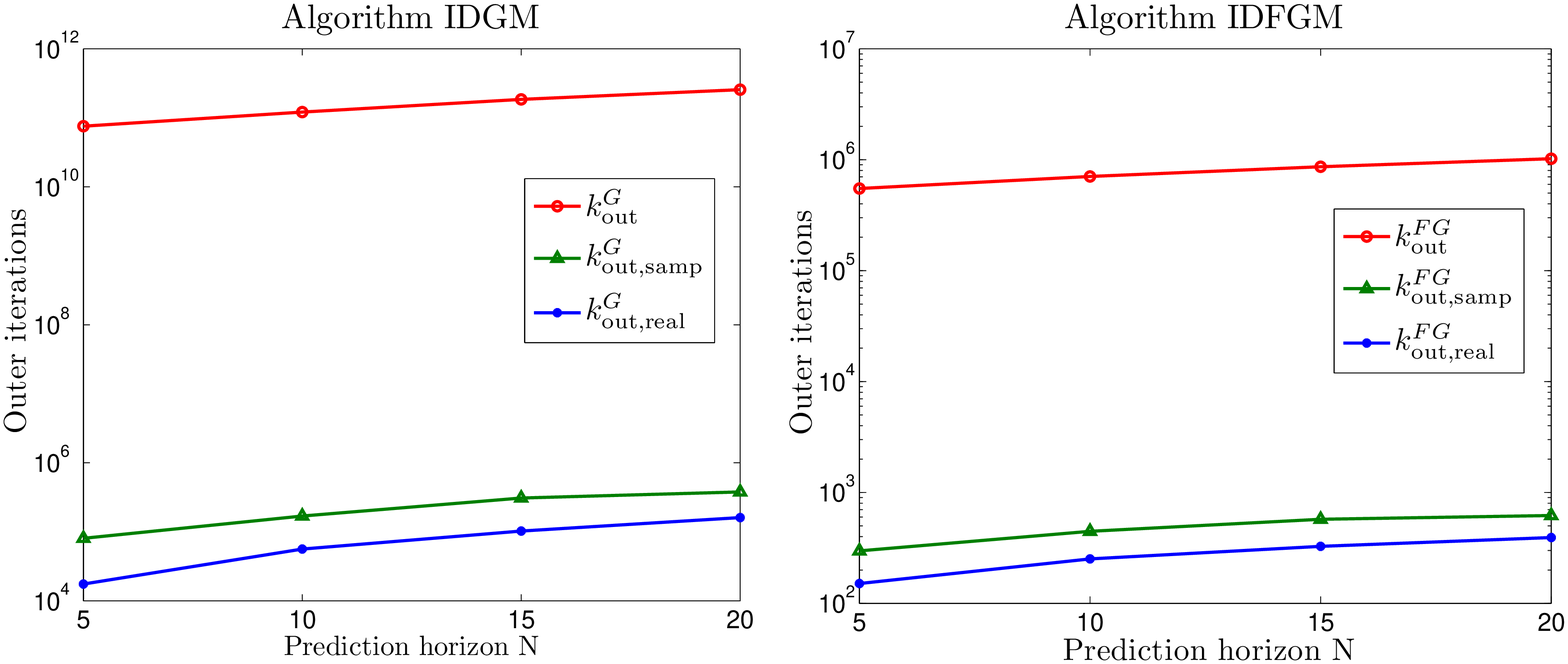}
\caption{Variation of $k_{\mathrm{out}}^i$,
$k_{\mathrm{out,samp}}^i$ and $k_{\mathrm{out,real}}^i$
($i=\set{G;FG}$) for Algorithm \eqref{iter_outer} (left) and
Algorithm \eqref{update_1} (right) w.r.t the prediction horizon $N$,
with accuracy $\eo=10^{-3}$.} \label{fig:outer_iterations_ball}
\end{figure}
We can observe that in practice Algorithm \eqref{update_1} performs
much better than Algorithm \eqref{iter_outer}. Also, we can notice
that the expected number of outer iterations
 $k_{\mathrm{out,samp}}^G$ and $k_{\mathrm{out,samp}}^{FG}$ obtained
from our derived bounds in Section \ref{sec_outer}  offer a good
approximation for the real number of iterations performed by the two
algorithms. Thus, these simulations show that our derived bounds in
Section \ref{sec_outer} are tight. On the other hand,
$k_{\mathrm{out}}^{FG}$ is about three orders of magnitude, while
$k_{\mathrm{out}}^G$ is about six orders of magnitude greater than
the real number of iterations.

In Figure \ref{fig:states_inputs} we also plot the evolution of the
states and inputs over the simulation horizon for a prediction
horizon $N = 5$ and an outer accuracy $\eo = 10^{-3}$. We observe
that the system is driven to the equilibrium point. Since we
obtained similar trajectories for the states and inputs  with both
algorithms, we present only the results for Algorithm
\eqref{update_1}.
\begin{figure}[h!]
\vskip-0.3cm
\centering\includegraphics[angle=0,height=4.0cm,width=12.5cm]{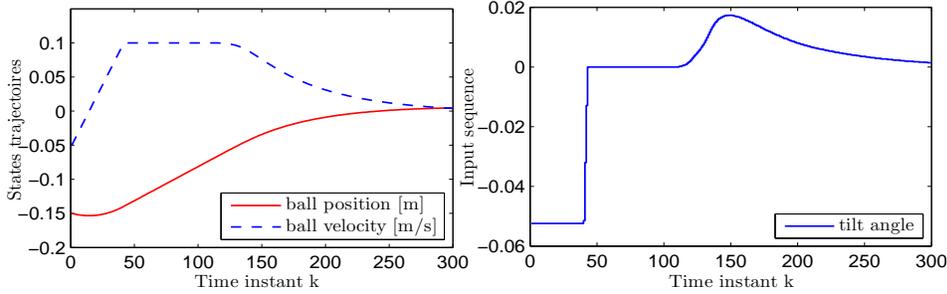}
\caption{The trajectories of the states and inputs  for a prediction
horizon $N = 5$ obtained using Algorithm \eqref{update_1} with
accuracy $\eo=10^{-3}$.} \label{fig:states_inputs}
\end{figure}

Since the number of outer iterations is also dependent on the way
the inner accuracy $\ei$ is chosen, we are also interested in  the
behavior of the two algorithms w.r.t to $\ei$. For this purpose we
apply Algorithms \eqref{iter_outer} and \eqref{update_1} for solving
the optimization problem \eqref{eq:Px} with a prediction horizon $N
= 20$, a fixed outer accuracy $\eo = 10^{-3}$ and varying $\ei$. In
Figure \ref{fig:inner_influence} we plot the average number of outer
iterations performed by the algorithms by taken $10$ random samples
 for the initial state $x \in X_N$.
\begin{figure}[h!]
\vspace{-0.1cm}
\centering\includegraphics[angle=0,height=4.0cm,width=12.5cm]{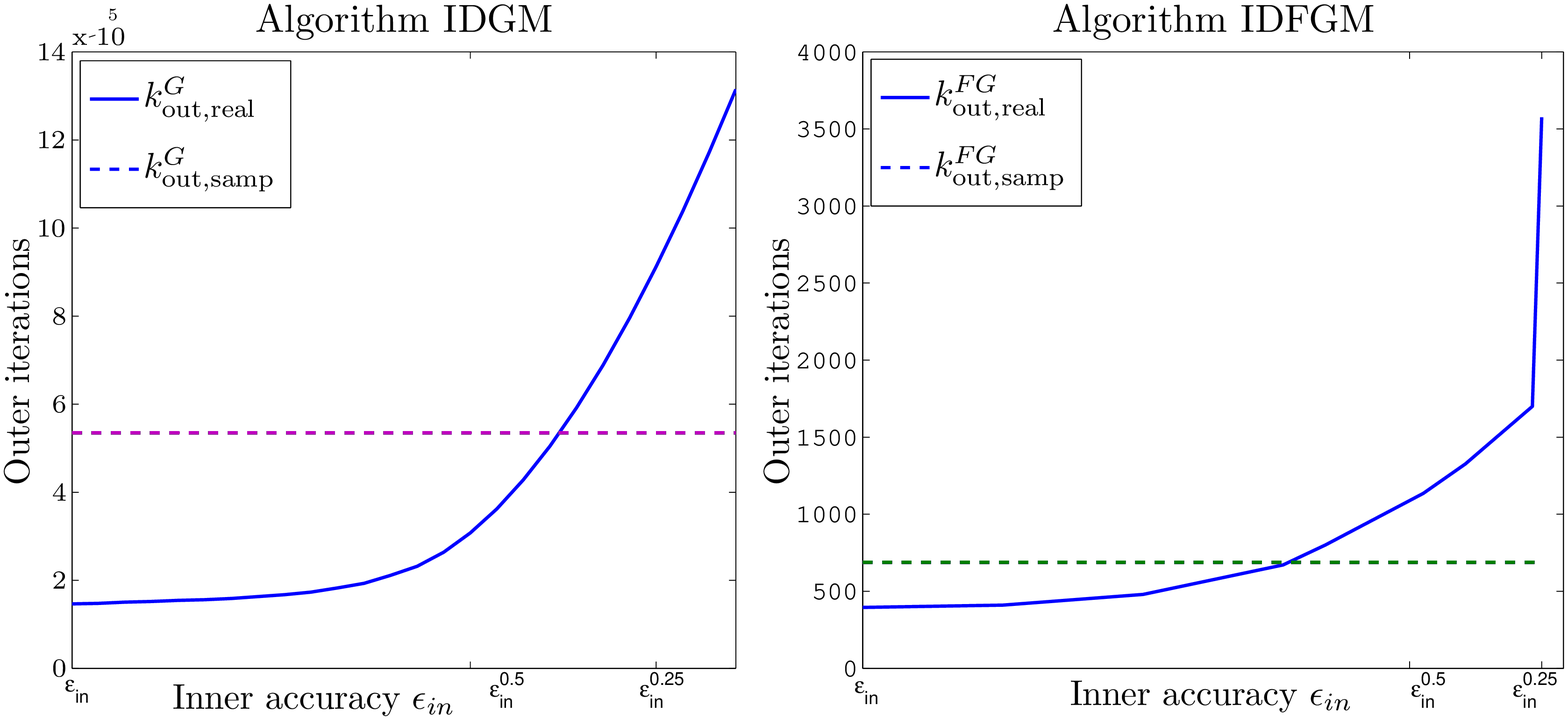}
\caption{The number of outer iterations performed by Algorithm
\eqref{iter_outer} (left) and Algorithm \eqref{update_1} (right)
with fixed outer accuracy $\eo=10^{-3}$ and different inner
accuracies $\ei$.} \label{fig:inner_influence}
\end{figure}
We observe that we can increase the inner accuracy $\ei$ derived in
Section \ref{sec_outer} up to a certain value and the algorithms
still perform a number of iterations less than the theoretical
bounds derived in Section \ref{sec_outer} for finding a suboptimal
solution. On the other hand, if the inner accuracy is too large, the
desired  suboptimality  cannot  be ensured in a finite number of
iterations. We see that Algorithm \eqref{iter_outer} is less
sensitive to the choice of the inner accuracy  $\ei$ than Algorithm
\eqref{update_1} due to the fact that Algorithm \eqref{update_1}
accumulates errors   (see Theorems \ref{theorem_primal} and
\ref{theorem_primal_optim}).

In conclusion, we notice from  simulations that on the one hand the
Algorithm \eqref{update_1} is faster than \eqref{iter_outer}, but on
the other hand that it is less robust. Thus, depending on the
requirements of the application,  one can choose between the two
algorithms.

%%%%%%%%%%%%%%%%%%%%%%%%%%%%%%%%%%%%%%%%%%%%%%%%%%%%%%%%%%%%%%%%%%%%%
%% 5.1.2. Oscillating masses.
%%%%%%%%%%%%%%%%%%%%%%%%%%%%%%%%%%%%%%%%%%%%%%%%%%%%%%%%%%%%%%%%%%%%%
\subsubsection{Oscillating masses}
The second example  is a system comprised of $M$ oscillating masses
connected by springs to each other and to walls on either sides,
having $2M$ states and $M-1$ inputs. For a detailed description of
the system, its parameters and constraints see \cite{WanBoy:10}. We
choose a quadratic stage cost with randomly generated positive
semidefinite matrices $Q \in \rset^{2M \times 2M}$, having
$\mathrm{rank}(Q) =M $, $R = 0.1 I_M$ and the final cost $P = Q$.

For this application we are  interested in  the CPU time. Thus,  we
consider only the Algorithm \eqref{update_1}, which is usually
faster than Algorithm \eqref{iter_outer}. In  simulations we vary
the number $M$ of masses and also the prediction horizon length $N$.
Further, we consider both formulations of the MPC problem: sparse QP
(i.e. we keep the states as variables) and dense QP (i.e. we
eliminate the states using the dynamics of the system). Our goal is
to compare the performances of Algorithm \eqref{update_1} and other
methods used in the framework of linear MPC.  Algorithm
\eqref{update_1} and \texttt{Gurobi 5.0.1} solver (\textbf{Gur1})
are used for solving the sparse formulation of the MPC problem.
\textbf{Alg. 1} in \cite{PatBem:12j,PatBem:12} and \texttt{Gurobi
5.0.1} (\textbf{Gur2}) are used for solving   the dense formulation
of the MPC problem. In the implementation of the Algorithm
\eqref{update_1} we consider an adaptive scheme in order to update
the penalty parameter $\rho$, similar to the one presented in
\cite{Ham:05}. Since the number of iterations is sensitive to the
choice of penalty parameter, we have also tuned the initial guess of
$\rho$. For each number of masses and prediction horizon, $50$
simulations were run starting from different random initial states,
We have considered the accuracy $\eo=10^{-3}$ and the stopping
criteria $\abs{f(\hat z_k ) - f^{*}}$ and $\norm{A \hat z_k - b}$
less than accuracy $\eo$.
\begin{table}[!h]
\begin{center}
\newcommand{\cell}[1]{{\!\!\!}#1{\!\!\!}}
\begin{scriptsize}
\caption{The average and maximum CPU time [s] (number of iterations)
for Algorithm \eqref{update_1}, \textbf{Alg. 1} in
\cite{PatBem:12j}, \texttt{Gurobi} solver for sparse form
(\textbf{Gur1}) and \texttt{Gurobi} solver for condensed form
(\textbf{Gur2}).} \label{table:oscillating_masses}
\begin{tabular}{|r|r|r|r|r|r|r|r|r|r|} \hline
M & N & \multicolumn{2}{|c|}{\textbf{IDFGM}} &
\multicolumn{2}{|c|}{\textbf{Gur1} } &
\multicolumn{2}{|c|}{\textbf{Alg. 1}} &
\multicolumn{2}{|c|}{\textbf{Gur2}} \\ \hline \cline{3-10} & & avg & max & avg & max & avg & max & avg & max    \\ \hline

\cell{5} &  \cell{5}  & \cell{0.03 (31)} & \cell{0.04 (33)} &
\cell{0.007 (10)}& \cell{0.008 (11)}&  \cell{0.05 (441)} &
\cell{0.07 (604)} & \cell{0.005 (9)} & \cell{0.008 (11)} \\ \hline

\cell{5}    & \cell{10} & \cell{0.07 (36)} & \cell{0.10 (51)} &
\cell{0.009 (11)}& \cell{0.010 (12)}&  \cell{0.13 (924)} &
\cell{0.18 (1331)} & \cell{0.007 (12)} & \cell{0.008 (13)} \\ \hline

\cell{5}    & \cell{20} & \cell{0.13 (65)}   & \cell{0.22 (110)}&
\cell{0.016 (11)}& \cell{0.017 (12)}&  \cell{0.33 (1199)} &
\cell{0.65 (2383)} & \cell{0.038 (12)} & \cell{0.043 (13)} \\ \hline

\cell{10}   & \cell{5}  & \cell{0.10 (28)}     & \cell{0.12 (30)} &
\cell{0.013 (10)}& \cell{0.014 (12)}&  \cell{0.24 (1611)} &
\cell{0.33 (2193)} & \cell{0.007 (10)} & \cell{0.008 (11)} \\ \hline

\cell{10}   & \cell{10} & \cell{0.25 (47)} & \cell{0.38 (72)}&
\cell{0.027 (11)}& \cell{0.028 (13)}& \cell{0.77 (2552)} &
\cell{1.34 (4449)} & \cell{0.037 (12)} & \cell{0.041 (13)} \\ \hline

\cell{10}   & \cell{20} & \cell{0.64 (70)}   & \cell{1.25 (135)} &
\cell{0.051 (12)}& \cell{0.055 (13)}&  \cell{2.75 (2331)} &
\cell{5.69 (4698)} & \cell{0.156 (10)} &  \cell{0.168 (12)} \\
\hline

\cell{20} &  \cell{5} &  \cell{0.23 (42)}   &  \cell{0.34 (64)}   &
\cell{0.039 (11)}   &  \cell{0.043 (12)}   &  \cell{0.99 (3066)}   &
\cell{1.45 (4481)} & \cell{0.020 (11)} &   \cell{0.025 (13)} \\
\hline

\cell{20} &  \cell{10} & \cell{1.54 (98)}    &  \cell{2.90 (193)} &
\cell{0.078 (12)} & \cell{0.084 (13)}   &  \cell{7.60 (5067)}   &
\cell{18.30 (11953)} & \cell{0.105 (12)} &  \cell{0.115 (13)} \\
\hline

\cell{20}   & \cell{20} &  \cell{8.2 (356)}   & \cell{14.9 (646)} &
\cell{0.230 (12)} & \cell{0.770 (13)}   & \cell{57.6 (12581)}    &
\cell{84.7 (18504)} & \cell{1.300 (12)}   &  \cell{2.120 (12)} \\
\hline
\end{tabular}
\end{scriptsize}
\end{center}
\end{table}

We can observe from Table \ref{table:oscillating_masses} that
Algorithm \eqref{update_1} outperforms \textbf{Alg. 1} in
\cite{PatBem:12j}, especially when the dimension of the problem
increases. On the other hand, we can notice that the solver
\texttt{Gurobi 5.0.1} performs much faster than our algorithm, since
the MPC problem is  sparse. However, the CPU times of the two
algorithms are comparable in the case of dense QP problems (see next
section).

%%%%%%%%%%%%%%%%%%%%%%%%%%%%%%%%%%%%%%%%%%%%%%%%%%%%%%%%%%%%%%%%%%%%%%%%%%%%%%%%%%%%%%%%%%%%
%% 5.2. Random quadratic programming problems.
%%%%%%%%%%%%%%%%%%%%%%%%%%%%%%%%%%%%%%%%%%%%%%%%%%%%%%%%%%%%%%%%%%%%%%%%%%%%%%%%%%%%%%%%%%%%
\subsection{Random quadratic programming problems}\label{subsec_random_numerical}
In this section we compare the performance, in terms of CPU time, of
Algorithms \eqref{iter_outer} and \eqref{update_1} against some well
known QP solvers used for solving MPC problems:
\texttt{quadprog}~(Matlab R2008b), \texttt{Sedumi}~1.3,
\texttt{Cplex}~12.4 (IBM ILOG) and \texttt{Gurobi} 5.0.1.

We consider random QP problems of the form
\begin{equation*}
\min \limits_{\mathrm{lb} \leq z \leq \mathrm{ub}} \set{0.5z^TQz +
q^Tz: \;\; \text{s.t.} \;  Az = b},
\end{equation*}
where matrices $Q \in \rset^{r \times n}$ and $A \in
\rset^{\lceil\frac{n}{2}\rceil \times n}$ are taken from a normal
distribution with zero mean and unit variance. Matrix $Q$ is then
made positive semidefinite by transformation $Q \leftarrow Q^T Q$,
having $\mathrm{rank}(Q)$ ranging from $0.5n$ to $0.9n$. Further,
$\mathrm{ub} = - \mathrm{lb} = 1$ and $b$ is taken from a uniform
distribution.

We plot in Figure \ref{fig:qp_cpu} the average CPU time for each
solver, obtained by solving $50$ random QP's for each dimension $n$,
with an accuracy  $\eo=10^{-3}$ and the stopping criteria
$\abs{f(\hat z_k ) - f^{*}}$ and $\norm{A \hat z_k - b}$ less than
accuracy $\eo$. In the case of Algorithm \eqref{iter_outer}, at each
outer iteration we let the algorithm perform only $k_{\mathrm{in}}^G
= 50$ inner iterations. For the Algorithm \eqref{update_1} we
consider two scenarios: in the first one, we let the algorithm to
perform only $k_{\mathrm{in}}^{FG}= 100$ inner iterations, while in
the second one we use the theoretic number of inner iterations
obtained in Section \ref{subsec_total} (see \eqref{eq_kinfastgrad}).
As described previously, in our algorithms we consider an adaptive
scheme for updating the penalty parameter $\rho$, similar to the one
presented in \cite{Ham:05}. We can observe that even if the
Algorithms \eqref{iter_outer} and \eqref{update_1} are well suited
for embedded applications, i.e. the implementation of the iterates
is very simple, the iteration complexity is low and also the number
of iterations for finding an approximate solution can be easily
predicted, the computational time is comparable with the one of the
other solvers used in the context of MPC.
\begin{figure}[h!]
\vskip-0.2cm \centering\includegraphics[angle=0,height =
4.5cm,width=12.5cm]{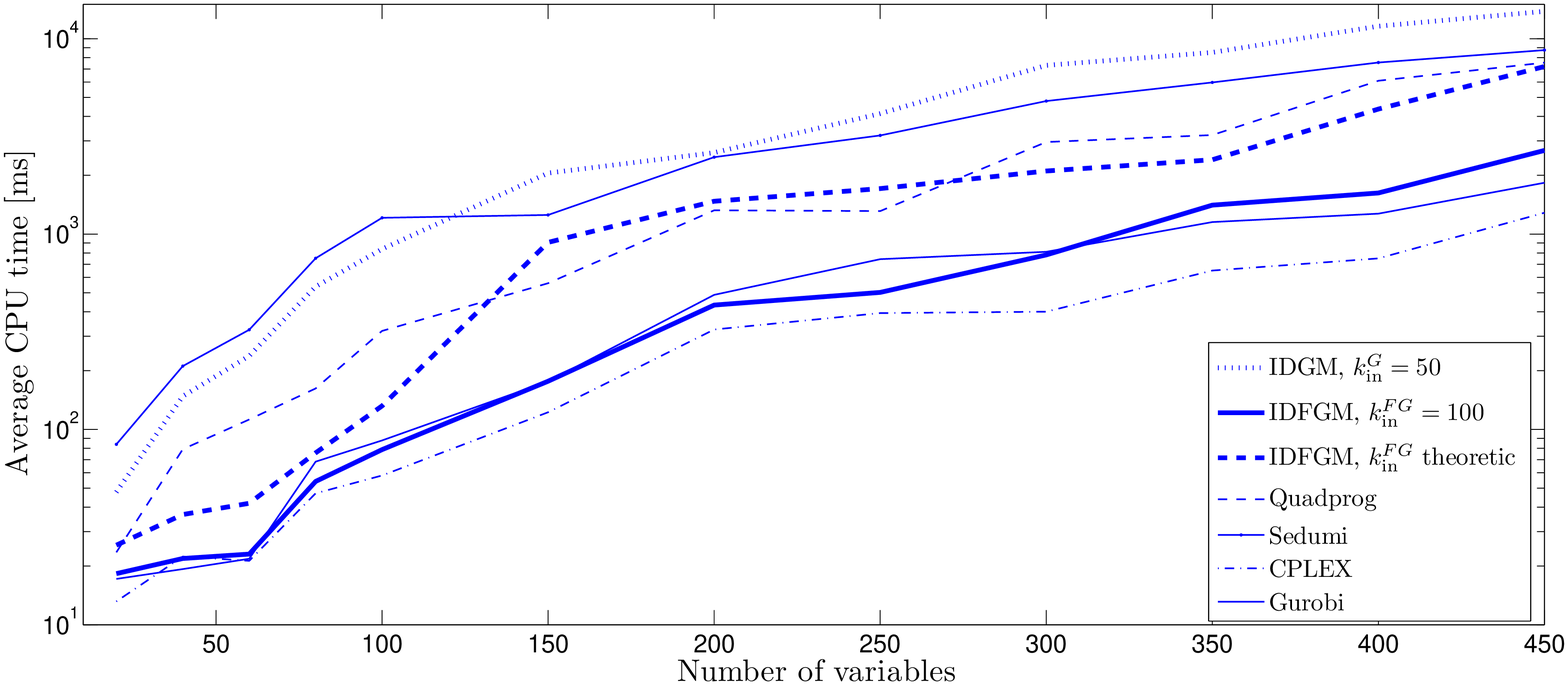}  \caption{Average CPU time for
solving QP problems of different sizes.} \label{fig:qp_cpu}
\end{figure}
Although the obtained averaged CPU times are comparable, we cannot
compare the exact computation complexity since for this purpose an
equivalence between different stopping criteria of each solver
should be studied, like e.g. the maximum violation of the
constraints.

%%%%%%%%%%%%%%%%%%%%%%%%%%%%%%%%%%%%%%%%%%%%%%%%%%%%%%%%%%%%%%%%%%%%%%%%%%%%%%%%%%%%%%%%%%%%%%%%%%%%%%
%%% 6. Conclusions.
%%%%%%%%%%%%%%%%%%%%%%%%%%%%%%%%%%%%%%%%%%%%%%%%%%%%%%%%%%%%%%%%%%%%%%%%%%%%%%%%%%%%%%%%%%%%%%%%%%%%%%
\section{Conclusions}\label{sec_con}
Motivated by  MPC problems for fast embedded linear systems, we have
proposed two dual  gradient based methods for solving the augmented
Lagrangian dual of a primal convex optimization problem with
complicating linear constraints.   We have moved the complicating
constraints in the cost using augmented Lagrangian framework  and
solved the dual problem using gradient and fast gradient methods
with inexact information. We  have solved the inner subproblems only
up to a certain accuracy, discussed the relations between the inner
and the outer accuracy of the primal and dual problems and derived
tight estimates on both primal and dual suboptimality and also on
feasibility violation. We have also discussed some implementation
issues of the new algorithms for embedded linear MPC problems and
tested them on  several examples.

\vskip0.3cm
\begin{footnotesize}
\noindent\textbf{Acknowledgements.} The research leading to these
results has received funding from: the European Union, Seventh
Framework Programme (FP7-EMBOCON/2007--2013) under grant agreement
no 248940; CNCS-UEFISCDI (project TE-231, no. 19/11.08.2010); ANCS
(project PN II, no. 80EU/2010); Sectoral Operational Programme
Human Resources Development 2007-2013 of the Romanian Ministry of
Labor, Family and Social Protection through the Financial
Agreement POSDRU/89/1.5/S/62557 and POSDRU/107/1.5/S/76909;
Research Council KUL: PFV/10/002 (OPTEC), GOA/10/09 - MaNet,
GOA/10/11; IOF/KP/SCORES4CHEM, G.0320.08, G.0377.09; IUAP - P7
(DYSCO); FP7 - SADCO (MC ITN-264735), ERC - ST - HIGHWIND
(259-166), Eurostars - SMART.
\end{footnotesize}

%%%%%%%%%%%%%%%%%%%%%%%%%%%%%%%%%%%%%%%%%%%%%%%%%%%%%%%%%%%%%%%%%%%%%%%%%%%%%%%%%%%%%%%%%%%%%
%+ References.
%%%%%%%%%%%%%%%%%%%%%%%%%%%%%%%%%%%%%%%%%%%%%%%%%%%%%%%%%%%%%%%%%%%%%%%%%%%%%%%%%%%%%%%%%%%%%
\bibliographystyle{plain}

\begin{thebibliography}{99}

\bibitem{Ber:99}
{\sc D.P. Bertsekas}, {\em Nonlinear Programming}, 2nd ed., Athena
Scientific, 1999.

\bibitem{BoyVan:04}
{\sc S. Boyd and L. Vandenberghe}, {\em Convex Optimization},
Cambridge University Press, 2004.

%\bibitem{BoyPar:11}
%{\sc S. Boyd, N. Parikh, E. Chu, B. Peleato and J. Eckstein}, {\em
%Distributed optimization and statistical learning via the
%alternating direction method of multipliers}, Found. Trends Mach.
%Learning, 3 (2011), pp. 1--124.

\bibitem{DevGli:12}
{\sc O. Devolder, F. Glineur and Y. Nesterov}, {\em Double
Smoothing Technique for Large Scale Linearly Constrained Convex
Optimization}, SIAM J. Optim., 22 (2012), pp. 702--727.

\bibitem{DevGli:11}
{\sc O. Devolder, F. Glineur and Y. Nesterov}, {\em First-order
methods of smooth convex optimization with inexact oracle}, CORE
Discussion Papers 02/2011, Universite Catholique de Louvain,
2011.

\bibitem{Ham:05}
{\sc A. Hamdi}, {\em Two-level primal–dual proximal decomposition
technique to solve large scale optimization problems}, Appl. Math.
Comput., 160 (2005), pp. 921--938.

\bibitem{Hes:69}
{\sc M.R. Hestenes}, {\em Multiplier and gradient methods},
Journal of Optimization and Application, 4 (1969), pp. 303--320.

\bibitem{HouFer:11}
{\sc B. Houska, H.J. Ferreau and M. Diehl}, {\em An auto-generated
real-time iteration algorithm for nonlinear MPC in the microsecond
range}, Automatica J. IFAC, 47 (2011), pp. 2279--2285.

\bibitem{JerCon:11}
{\sc J.L. Jerez, K.-V. Ling, G.A. Constantinides and E.C.
Kerrigan}, {\em Model predictive control for deeply pipelined
field-programmable gate array implementation: Algorithms and
circuitry}, IET Control Theory Appl., 6 (2012), pp. 1029--1041.

\bibitem{KogFin:11}
{\sc M. Kogel and R. Findeisen}, {\em Fast predictive control of
linear systems combining Nesterov's gradient method and the method
of multipliers}, Proc. of IEEE Conference on Decision and Control,
pp. 501 - 506, 2011.

\bibitem{KogFin:11i}
M. Kogel and R. Findeisen, ``A Fast Gradient method for embedded
linear predictive control'', in {\it Proc.  IFAC World Congress},
pp. 1362 - 1367, 2011.

\bibitem{LanMon:08}
{\sc G. Lan and R.D.C. Monteiro}, {\em Iteration-complexity of
first-order augmented Lagrangian methods for convex programming},
Technical Report, School of Industrial and Systems Engineering,
Georgia Institute of Technology, 2008.

\bibitem{NecSuy:08}
{\sc I. Necoara, J. Suykens}, {\em Application of a Smoothing
Technique to Decomposition in Convex Optimization}, IEEE Trans.
Automat. Control, 53 (2008), pp. 2674--2679.

\bibitem{NedOzd:09}
{\sc A. Nedic and A. Ozdaglar}, {\em Approximate primal solutions
and rate analysis for dual subgradient methods}, SIAM J. Optim, 19
(2009), pp. 1757--1780.

\bibitem{Nes:04}
{\sc Y. Nesterov}, {\em Introductory lectures on convex
optimization}, Springer, 2004.

\bibitem{Nes:04a}
{\sc Y. Nesterov}, {\em Smooth minimization of non-smooth
functions}, Math. Program., 103 (2005), pp. 127--152.

\bibitem{PatBem:12}
{\sc P. Patrinos and A. Bemporad}, {\em Simple and Certifiable
Quadratic Programming Algorithms for Embedded Linear Model
Predictive Control}, Proc. of IFAC Nonlinear Model Predictive
Control Conference, 2012.

\bibitem{PatBem:12j}
{\sc P. Patrinos and A. Bemporad}, {\em An accelerated dual
gradient-projection algorithm for embedded linear model predictive
control}, IEEE Trans. Automat. Control, 2012.

\bibitem{RaoWri:98}
{\sc C. V. Rao, S. J. Wright and J. B. Rawlings}, {\em Application
of interior-point methods to model predictive control}, J. Optim.
Theory Appl., 99 (1998), pp. 723--757.

\bibitem{RicMor:11}
{\sc S. Richter, M. Morari and C.N. Jones}, {\em Towards
Computational Complexity Certification for Constrained MPC Based
on Lagrange Relaxation and the Fast Gradient Method}, Proc. of
IEEE Conference on Decision and Control, 2011.

\bibitem{RicJon:12}
{\sc S. Richter, C.N. Jones and M. Morari}, {\em Computational
Complexity Certification for Real-Time MPC With Input Constraints
Based on the Fast Gradient Method}, IEEE Trans. Automat. Control,
57 (2012), pp. 1391--1403.

\bibitem{Roc:98}
{\sc R.T. Rockafellar and R. Wetz}, {\em Variational Analysis},
Springer-Verlag, 1998.


\bibitem{Roc:76}
{\sc R.T. Rockafellar}, {\em Augmented Lagrangian and Applications
of The Proximal Point Algorithm In Convex Programming}, Math.
Oper. Res., 1 (1976), pp. 97--116.

\bibitem{ScoMay:99}
{\sc P.O.M Scokaert, D.Q. Mayne and J.B. Rawlings}, {\em
Suboptimal model predictive control (feasibility implies
stability)}, IEEE Trans. Automat. Control, 44 (1999), pp.
648--654.

\bibitem{ValRos:11}
{\sc G. Valencia-Palomo and J.A. Rossiter}, {\em Programmable logic
controller implementation of an auto-tuned predictive control based
on minimal plant information}, ISA Transactions, 50 (2011), pp.
92--100.

\bibitem{Zometa}
{\sc P. Zometa, M. Kogel, T. Faulwasser and R. Findeisen}, {\em
Implementation Aspects of Model Predictive Control for Embedded
Systems}, Proc. of American Control Conference, 2012.

\bibitem{WanBoy:10}
{\sc Y. Wang and S. Boyd}, {\em Fast Model Predictive Control
Using Online Optimization}, IEEE Trans. Control Syst. Tech., 18
(2010), pp. 267--278.

\end{thebibliography}

\end{document}